 \documentclass[reqno, 11pt]{amsart}
\usepackage[percent]{overpic}
\usepackage{calc,graphicx,amsfonts,amsthm,amscd,amsmath,amssymb,enumerate,dsfont,mathrsfs,paralist,bbm,tikz}
\usepackage{subfig}
\usepackage{pdfsync}
\usepackage{stmaryrd}
\usepackage[numeric,initials,nobysame]{amsrefs}
\usepackage{marginnote}

\usepackage[colorlinks=true, pdfstartview=FitV, linkcolor=blue, citecolor=blue, urlcolor=blue,pagebackref=false]{hyperref}

\usepackage{booktabs}
\newcommand{\ie}{\hbox{\it i.e.\ }}

\reversemarginpar
\newlength\fullwidth
\numberwithin{equation}{section}

\DeclareMathSymbol{\leqslant}{\mathalpha}{AMSa}{"36} % nicer `smaller or equal'
\DeclareMathSymbol{\geqslant}{\mathalpha}{AMSa}{"3E} % nicer `larger or equal'
\DeclareMathSymbol{\eset}{\mathalpha}{AMSb}{"3F}     % nicer `emptyset'
\renewcommand{\leq}{\;\leqslant\;}                   % redef. of < or =
\renewcommand{\geq}{\;\geqslant\;}                   % redef. of > or =
\newcommand{\sumtwo}[2]{\sum_{\substack{#1 \\ #2}}} % sum with 2 lines
\def\1{\ifmmode {1\hskip -3pt \rm{I}} \else {\hbox {$1\hskip -3pt \rm{I}$}}\fi}

\newcommand{\var}{\operatorname{Var}}

\newcommand{\id}{\mathbbm{1}}

\newcommand{\trel}{T_{\rm rel}}

\newcommand{\D}{\Delta}

\renewcommand{\l}{\lambda}
\renewcommand{\L}{\Lambda}

\renewcommand{\l}{\lambda}
\renewcommand{\a}{\alpha}
\renewcommand{\d}{\delta}
\renewcommand{\t}{\tau}

\newcommand{\g}{\gamma}
\newcommand{\G}{\Gamma}

\newcommand{\e}{\varepsilon}

\renewcommand{\o}{\omega}
\renewcommand{\O}{\Omega}

\newcommand{\tc}{\thinspace |\thinspace}

\newtheorem{theorem}{Theorem}[section]
\newtheorem{lemma}[theorem]{Lemma}
\newtheorem{proposition}[theorem]{Proposition}
\newtheorem{corollary}[theorem]{Corollary}
\newtheorem{remark}[theorem]{Remark}

\newtheorem{claim}[theorem]{Claim}
\newtheorem{definition}[theorem]{Definition}
\newtheorem{maintheorem}{Theorem}

\newtheorem*{question*}{Question}

\newtheorem*{remark*}{Remark}
\newtheorem*{idefinition*}{Definition}
\newtheorem{example}{Example}

\newcommand{\cA}{\ensuremath{\mathcal A}}
\newcommand{\cB}{\ensuremath{\mathcal B}}
\newcommand{\cC}{\ensuremath{\mathcal C}}
\newcommand{\cD}{\ensuremath{\mathcal D}}

\newcommand{\cL}{\ensuremath{\mathcal L}}

\newcommand{\cP}{\ensuremath{\mathcal P}}

\newcommand{\cU}{\ensuremath{\mathcal U}}

\newcommand{\bbE}{{\ensuremath{\mathbb E}} }

\newcommand{\bbN}{{\ensuremath{\mathbb N}} }

\newcommand{\bbP}{{\ensuremath{\mathbb P}} }

\newcommand{\bbR}{{\ensuremath{\mathbb R}} }

\newcommand{\bbZ}{{\ensuremath{\mathbb Z}} }

\newcommand{\ex}{{Ext}_x}

\let\a=\alpha    \let\d=\delta  \let\e=\varepsilon
 \let\g=\gamma       \let\l=\lambda
      \let\o=\omega      
  \let\s=\sigma \let\t=\tau   
  
\let\D=\Delta   \let\G=\Gamma  \let\L=\Lambda 
\let\O=\Omega      

\renewcommand{\le}{\leq}

\begin{document}
%%*rimesso qui
\title[]{Towards a universality picture for the relaxation to
  equilibrium of kinetically constrained models} 
\author[F. Martinelli]{F. Martinelli}
\email{martin@mat.uniroma3.it}
\address{Dipartimento di Matematica e Fisica, Universit\`a Roma Tre, Largo S.L. Murialdo 00146, Roma, Italy}
\author[C. Toninelli]{C. Toninelli}
\email{cristina.toninelli@upmc.fr}
\address{Laboratoire de Probabilit\'es et Mod\`eles Al\`eatoires
  CNRS-UMR 7599 Universit\'es Paris VI-VII 4, Place Jussieu F-75252 Paris Cedex 05 France}
\thanks{This work has been supported by the ERC Starting Grant 680275 MALIG
 }
 \begin{abstract}
Recent years have seen a great deal of progress in our understanding
of bootstrap percolation models, a particular class of \emph{ monotone
  cellular automata}. In the two-dimensional lattice $\bbZ^2$ there is now a quite satisfactory understanding of their evolution starting from a random initial condition, with a strikingly beautiful universality picture for their critical behavior. 
Much less is known for their non-monotone stochastic counterpart, namely \emph{kinetically constrained models} (KCM). In KCM each vertex is resampled (independently) at rate one by tossing a $p$-coin iff it can be infected in the next step by the bootstrap model. In particular an infection can also heal, hence the non-monotonicity.
Besides the connection with bootstrap percolation, KCM have an
interest in their own as they feature some of the most striking features of the liquid/glass transition, a major and still largely open problem in condensed matter physics.
In this paper  we pave the way towards proving universality results for the characteristic time scales of KCM. Our novel and general approach gives the right tools to establish a close connection between the critical scaling of characteristic time scales for KCM and the scaling of the critical length in critical bootstrap models. When applied to the Fredrickson-Andersen k-facilitated models in dimension $d\ge 2$, amongst the most studied KCM, and to the Gravner-Griffeath model, our results are close to optimal.% Our novel and
% general approach establishes a close connection between the critical
% scaling of characteristic time scales for KCM and the scaling of the
% critical length in critical bootstrap models. Although the full proof
% of universality for KCM is deferred to a forthcoming paper, here we
% apply our general method to the Fredrickson-Andersen $k$-facilitated
% models, amongst the most studied KCM, and to the Gravner-Griffeath
% model. In both cases our results are close to optimal. 
\end{abstract}
\maketitle

\section{Introduction}
In recent years remarkable progress has been obtained in understanding the behaviour of a  particular class of monotone cellular automata known as \emph{bootstrap percolation}. A general bootstrap cellular automaton \cite{BSU} is specified by its \emph{update family} $\mathcal U=\{U_1,\dots,U_m\}$ of finite subsets of $\mathbb Z^d\setminus 0$. Once $\cU$ is given, the  $\mathcal U$-bootstrap percolation process is as follows.
Given a set $A \subset \mathbb Z^d$ of initially infected vertices, set $A_0 = A$, and
define recursively for each $t \in\mathbb N$
\begin{equation}
  \label{eq:2}
A_{t+1}=A_t \cup \{x\in \mathbb Z^d:\ x+U_k\subset A_t\ \text{ for some }
k\in (1,\dots m)\}.
\end{equation}
In other words a vertex $x$ becomes infected at time $t+1$ if the
translate by $x$ of at least one element of the update family is
already entirely infected at time $t$, and infected vertices remain
infected forever. We write $[A]_\cU := \bigcup_{t \ge 0} A_t$ for the
\emph{closure} of $A$ under the $\cU$-bootstrap process.

 A much studied example is the classical
$r$-neighbour model (see \cite{BBD-CM} and references therein) in which a vertex gets infected if \emph{at least} $r$ among its nearest neighbours are infected, namely the update family is formed by all the $r$-subsets of the set of the nearest neighbours of the origin.

A central  problem for bootstrap models is their long-time evolution
when the initial infected set $A_0$ is $\bbP_q(\cdot)$-random, \ie each vertex of
$\mathbb Z^d$, independently from the other vertices, is initially
declared to be infected with probability $q\in (0,1)$. A key quantity
is then the \emph{critical percolation threshold}  
\[
q_c(\mathcal U):=\inf\{q: \bbP_q([A]=\bbZ^d)=1\}. 
\] 
Two closely related quantities are $T_c(q;\cU)$ and $L_c(q;\cU)$
defined as follows.  
\begin{definition}
\label{def:scale critiche}Let $\tau_{BP}=\min\{t:\ 0\in [A]_t\}$ be the infection time of
the origin. Then
\[
T_c(q;\cU)=\inf\{t\ge 0:\ \bbP_q(\tau_{BP}\ge t)\le 1/2\}. 
\] 
To define $L_c(q;\cU)$, let us consider the bootstrap percolation
process on the $d$-dimensional
torus $\bbZ^d_n\subset \bbZ^d$ of linear size $n$, and let 
$q_c(n;\mathcal U)$ be the smallest $q$ such that with probability at
least $1/2$ the whole torus is 
eventually infected. Then 
\begin{equation}
  \label{eq:3}
L_c(q,\mathcal U):=\min\{n:q_c(n,\mathcal U)\le q\}.
\end{equation}
\end{definition}
In \cites{BSU,BBPS,BD-CMS} beautiful universality results for general
$\mathcal U$-bootstrap percolation processes in two dimensions
satisfying $q_c(\cU)=0$ have been established,
yielding in particular the sharp scaling behaviour of  $T_c(q;\cU), L_c(q;\cU)$ as $q\to 0$. For a nice review of these
results we refer the reader to \cite{Robsurvey}*{Section 1}. It
follows in particular \cite{BD-CMS}*{Theorems 1.4, 1.5} that in two dimensions
\[
0< \liminf_{q\to 0}\frac{\log(T_c(q;\cU))}{\log(L_c(q;\cU))}\le \limsup_{q\to
  0}\frac{\log(T_c(q;\cU))}{\log(L_c(q;\cU))}<+\infty,
\] 
and in this sense one can say that $T_c(q;\cU)$ and $ L_c(q;\cU)$ have the
same scaling behavior.

A quite natural stochastic counterpart of bootstrap percolation models are particular interacting particle systems known as \emph{kinetically constrained models} (KCM). 
Given a $\mathcal U$-bootstrap model, the associated KCM is the continuous-time reversible Markov process on $\O=\{0,1\}^{\mathbb Z^d}$ constructed as follows. 
Denote by
$\o\in \O$ the current configuration of the process and call a vertex
$x$ \emph{infected} if $\o_x=0$. Then each
vertex $x$, with rate one and independently across $\bbZ^d$, is
resampled by tossing a $p$-coin ($\text{Prob}(1)=p$)  iff the
translate by $x$ of at least one element of the update family $\cU$ is
already entirely infected for $\o$. In other words the state $\o_x$ of
the vertex $x$ is allowed to be resampled iff it was infectable in $\o$
by the
bootstrap process \cite{CMRT}. 

It is easy to check that such a process is
reversible w.r.t. the Bernoulli($p$) product measure $\mu$ on $\O$.
Notice that if $q:=1-p\ll 1$, it is very unlikely for a vertex to
become infected (even if it would have been infected by the bootstrap
process). Observe moreover that infected vertices may heal. The latter feature implies, in particular, that the KCM is not \emph{monotone/attractive}, a fact that rules out several powerful tools from interacting particle systems theory like monotone coupling and censoring.

Besides the connection with cellular automata, KCM are of interest in their own. 
They have been in fact introduced in the physics literature in the '80's to model the liquid/glass transition, a major and still largely open problem in condensed matter physics \cite{Berthier-Biroli}.
Extensive numerical simulations indicate that KCM display a remarkable  \emph{glassy} behavior, including heterogeneous dynamics, the occurrence of ergodicity breaking transitions, multiple invariant measures and  anomalously long-time scales  (see for example \cite{GarrahanSollichToninelli} and references therein).

It has been proved in \cite{CMRT} that the KCM undergoes an ergodicity breaking
transition at $q_c(\cU)$
and a major problem, both from  the physical and mathematical point of view, is to determine the precise divergence of its characteristic time scales 
when $q\downarrow q_c$. 
A natural time scale  is the mean hitting time $\bbE_\mu(\t_0)$, where $\bbE_\mu(\cdot)$ is the average w.r.t. to the law of the stationary KCM process with initial law $\mu$ and 
\[
\tau_0=\inf\{t\ge 0:\ \o_0(t)=0\}.
\] 
For all those KCM whose update family $\cU$ satisfies $q_c(\cU)=0$, one can then ask whether the scaling of $\bbE_\mu(\t_0)$ as
$q\downarrow 0$ is related to that of $T_c(q;\cU),L_c(q;\cU)$. It is
possible to prove (see Lemma \ref{LEM:PARIS}) that there exists
$\d=\d(\cU)\in (0,1)$ such that, for all $q$ small enough,
\begin{equation}
  \label{eq:rev1}
\bbE_\mu(\t_0)\ge \d \bbE_q(\t_{BP})\ge \frac{\d}{2} T_c(q;\cU).
\end{equation}
So far, the best \emph{general upper bound} on $\bbE_\mu(\tau_0)$ is  a much poorer one of the form \cite{CMRT}
\[
\bbE_\mu(\tau_0)\leq e^{O(L_c(q;\cU)^d)}.
\] 
Although this bound has been greatly improved for special choices of the update family $\cU$, yielding in some cases the correct behavior (cf.  \cites{CFM,Tree1,CFM2}), for general KCM and contrary to the situation of bootstrap percolation in two dimensions, there is yet no universality picture for the scaling of  $\bbE_\mu(\tau_0)$.

The present paper represents the first step of a general project concerning KCM with update family $\cU$ satisfying $q_c(\cU)=0,$ with the aim of establishing universality results on the scaling 
of $\bbE_\mu(\tau_0)$ as $q\to 0$ analogous to those proved within bootstrap percolation. 

At the beginning of this program, in \cite{Robsurvey}*{Section 2} some
conjectures were put forward, jointly with us, on the scaling of
$\bbE_\mu(\t_0)$. In particular, it was suggested that for KCM it is
necessary to introduce a more refined classification of the
universality classes in order to take into account the effect of the
possible presence of \emph{energy barriers} in the dynamics. By
energy barriers we mean very unlikely states with an anomalous amount
of infection which are typically visited by the stationary KCM process
before infecting the origin. More specifically, it was argued that
energy barriers could induce a very different scaling of
$\bbE_\mu(\t_0)$ w.r.t. to that of $T_c(q;\cU)$ for all those models
for which the characteristic \emph{bootstrap percolation critical
  droplets} are constrained to move inside a cone. Examples include
the two-dimensional East \cite{CFM} with $\cU$ consisting of the
$1$-subsets of $\cup_{i=1}^d\{-\vec e_i\}$ and Duarte-KCM model
\cites{Duarte,Mountford,BCMS-Duarte} where $\cU=2\text{-subsets of}
\{\vec -e_1,\pm \vec e_2\}$. Significant progresses in this direction
have been obtained after this work has been completed \cite{MMT2}. 

The main purpose of this paper is twofold. Firstly, we envisage
a general and novel approach to prove Poincar\'e inequalities for KCM,
with the ultimate goal of finding the exact scaling of $\mathbb
E_\mu(\tau_0)$ for a
very large class of update families $\cU$ with $q_c(\cU)=0$. 
For example, building upon the strategy and techniques developed in Sections \ref{CPoincare} and
\ref{application}, the following result has been recently established.
\begin{maintheorem}[\cite{MMT}]For the so-called \emph{critical $\alpha$-unrooted} KCM in two dimensions
(see \cite{Robsurvey} for the appropriate definition), 
\begin{equation}
  \label{eq:mainres}
\bbE_{\mu}(\tau_0)=O(L_c(q;\cU)^{\beta(q)}),\quad \beta(q)=\text{poly}(\log\log L_c(q;\cU))\quad \text{as }q\to 0.
\end{equation}
\end{maintheorem}
Secondly we want to greatly improve the existing upper bounds on
$\bbE_{\mu}(\tau_0)$ for the most studied KCM on $\bbZ^d$, in any
dimension $d\ge 2$, namely the
Fredrickson-Andersen $k$-facilitated model (FA-kf) \cite{FH}. Its update family
consists of the $k$-sets of the neighbors of the origin, and therefore
its associated bootstrap percolation version is the well known $k$-neighbour
model. We also test the flexibility of our techniques by briefly
analysing the
kinetically constrained version of the well known Gravner-Griffeath
bootstrap percolation model on $\bbZ^2$ \cites{GG2,GG}. In this case $\mathcal U$ consists
of the $3$-subsets of the set formed by the neighbours of the origin together with the vertices $(\pm
2,0)$ and it is known that the bootstrap process features a striking
anisotropy. In both  cases our main result (see Theorem
\ref{thm:main1}) establishes a tight connection between $\bbE_\mu(\t_0)$ and $L_c(q;\cU)$.

\subsection{Main results and plan of the paper}
In section \ref{CPoincare}, after introducing the relevant notation and motivated by the connection between $\bbE_\mu(\t_0)$ and the Poincar\'e inequality, we prove our first main result (Theorem  \ref{thm:1}). It establishes a (constrained) Poincar\'e inequality for very general KCM satisfying a rather flexible condition involving the range of the update family $\cU$ and the probability that an update is feasible. Constrained Poincar\'e inequalities for KCM, implying a positive spectral gap and exponential mixing, have already been established \cite{CMRT}, mainly using the so-called halving method. Here, inspired by our previous analysis of KCM on trees \cites{Tree2,Tree1}, we develop an alternative method which, besides being more natural and direct, applies as well to update families with a large (depending on $q$) or infinite number of elements. As an example, in section \ref{supercritical} we prove a Poincar\'e inequality for the KCM for which
the constraint requires that the oriented neighbours of the to-be-updated vertex belong to an infinite cluster of infected vertices. 

Section \ref{application}, and its main outcomes summarised in Corollary \ref{rhs 9bis}, is somehow the core of the work.  By applying Theorem \ref{thm:1} together with a renormalisation argument and canonical-paths arguments, we prove a sharp bound on the best constant in the Poincar\'e inequality for general KCM. This bound involves the probability of occurrence of a \emph{critical droplet} (in the bootstrap percolation language) together with certain \emph{congestion constants} related to the cost of moving around the droplet.   
In this section we made an effort to keep the framework as general as possible, in order to construct a very flexible tool that can be applied to any choice of constraints in any dimensions.

In section \ref{sec:models} we introduce the Fredrickson-Andersen
$k$-facilitated (FA-kf) and the Gravner-Griffeath kinetically constrained
(GG-KCM) models and state our main result Theorem \ref{thm:main1} for the scaling of $\bbE(\t_0(\bbZ^d;\cU))$ in these cases. 
Finally in section \ref{proofmainthm} we prove Theorem \ref{thm:main1} by bounding (model by model) the congestion constants appearing in the key  inequality 
of Corollary \ref{rhs 9bis}.

\section{A constrained Poincar\'e inequality for product measures}
\label{CPoincare}
\subsubsection{Notation}
For any integer $n$ we will write $[n]$ for the set
$\{1,2,\dots,n\}$. Given $x=(x_1,\dots,x_d)\in \bbZ^d$ we denote its $\ell^1$-norm by
$\|x\|_1=\sum_{i=1}^d|x_i|$ and by $d_1(\cdot,\cdot)$ the associated
distance function. Given two vertices $x\neq y$ we will say that $x$
precedes $y$ and we will write $x\prec y$ if $x_i\le y_i$ for all
$i\in [d]$.
The collection $\cB=\{\vec e_1, \vec e_2 , \dots ,\vec e_d\}$ will
denote the canonical basis of
$\bbZ^d$. Given a set $\L \subset \bbZ^d$ we define its
\emph{external boundary} as the set
$$ 
\partial \L=\{ y \in \bbZ^d \setminus \L\,:\, \exists \, x\in \L
\text{ with } \|x-y\|_1=1\}\,.
$$
% Finally, for any $\L\subset \bbZ^d$, a configuration
% $\s\in \O_{\partial_E \L}$ will be referred to as a \emph{boundary
%   condition}.
\subsubsection{The probability space}
Given a finite set $S$ and $\L\subseteq \bbZ^d$, we will denote by $\O_\L$ the product space
$S^\L$ endowed with the product topology. Given $V\subset \L$
and $\o\in \O_\L$ we will write $\o_{V}$ for the
restriction of $\o$ to $V$. Finally we will denote by $\mu_\L$ the
product measure $\mu_\L=\otimes_{x\in
  \L}\ \hat \mu_x$ on $\O_\L$ where, $\forall
x\in \bbZ^d$, we set $\hat\mu_x= \hat \mu$ with $\hat \mu$ a
probability measure on $S$ which w.l.o.g. we assume to be positive. Expectation and variance w.r.t. $\mu_\Lambda$ are denoted
by $\bbE_\L(\cdot),\ \var_\L(\cdot)$ respectively. If $\L = \bbZ^d$
the subscript $\L$ will be dropped from the notation. 

In several applications the probability space $\left(S,\hat \mu\right) $ will be the ``particle space''
$S=\{0,1\}^V$ where $V$ is a finite subset (a ``block'' as it is sometimes
called) of $\bbZ^d$ and $\hat \mu=\otimes_{x\in
  V} B(p)$, $B(p)$ being the $p$-Bernoulli measure. 

\subsubsection{The constraints} \label{sec:constraints}For each $x\in \bbZ^d$ let
$\D_x\subset \bbZ^d\setminus \{x\}$ be a finite set, let $\cA_x$ be
an event depending on the variables $\{\o_y\}_{y\in \D_x}$ and let
$c_x$ be its indicator
function. By construction $c_x$ does not depend on $\o_x$. In the
sequel we will refer to $c_x$ as the
\emph{constraint at $x$} and to $\e_x:=\mu(1-c_x)=\mu(\cA_x^c)$ and
$\D_x$ as its \emph{failure probability} and \emph{support} respectively. 
In our approach based on a martingale decomposition of the variance
$\var(f)$ of any local function $f:\O\mapsto \bbR$, a key role is
played by constraints satisfying the following \emph{exterior condition}.
\begin{definition}[Exterior condition]
\label{ext-cond}Given an
exhausting  collection of
subsets  $\{V_n\}_{n\in \bbZ}$ of $\bbZ^d$ 
%footnote
(\ie $V_n\subset V_{n+1}$ for all $n$ and $\cup_nV_n=\bbZ^d$),
let $\cL_n:=V_{n}\setminus V_{n-1}$ be the $n^{th}$-shell and, for any $x\in
\cL_n,$ let the
\emph{exterior} of $x$ be the set $\ex:=\cup^\infty_{j=n+1}\cL_j$.
We then say that the family of constraints $\{c_x\}_{x\in \bbZ^d}$
satisfies the exterior condition w.r.t. $\{V_n\}_{n=-\infty}^\infty$ if 
$\D_x\subset \ex$ for all $x$. We will say that $\{c_x\}_{x\in \bbZ^d}$
satisfies the exterior condition if there exists a family of sets
$\{V_n\}_{n\in \bbZ}$ as above such that $\{c_x\}_{x\in \bbZ^d}$
satisfies the exterior condition w.r.t. $\{V_n\}_{n\in \bbZ}$.  
\end{definition}
\begin{example}
\label{example:1}A concrete example of a class of constraints satisfying the exterior condition and
entering in the applications to
kinetically constrained models is as
follows. Fix a
vertex $z\succ 0$ and let $\cL_0=\{x\in \bbZ^d:\ \langle
x,z\rangle=0\},$ where $\langle\cdot,\cdot\rangle$ is the usual scalar
product and $x,z$ are treated as vectors in $\bbR^d$. For $j\in \bbZ$ let
$\cL_j=\cL_0+j\d\vec z$ where $\d=\sup\{\d'>0:\ (\cL_0+\d'\vec z)\cap
\bbZ^d=\emptyset\}$ (cf. Figure \ref{fig:exterior}) and let $V_n=
\cup_{j=-\infty}^n\cL_j$. The above construction defines the exhausting  collection of
subsets  $\{V_n\}_{n\in \bbZ}$. 

Let now $G\subset S$ be an single-site
event and let $\cU=(U_1,\dots,U_m)$ be a finite family of
finite subsets of 
the half-space $\{x\in \bbZ^d:\ \langle x,z\rangle >0\}=\cup_{j=1}^\infty \cL_j$. Then
we define $c_0(\o)$ as the indicator of the event that there exists $U\in \cU$
such that $\o_x\in G$ for all $x\in
U$. The
constraint $c_x$ at any other vertex $x$ is obtained by translating the
above construction by $x$. For example in $d=2$ one could take $S=\{0,1\}$, $G=\{0\}$, 
$z=(1,1)$, $m=1$ and $U=\{(0,1),(1,0)\}$, a case known as the
North-East model (cf. e.g. \cite{CMRT}). % In all the applications
% discussed in this paper $z=(1,\dots,1)$ but in order to
% prove the universality results discussed in the introduction more
% general choices of $z$ will be necessary. 

\end{example}
\begin{figure}
\begin{tikzpicture}[scale=0.8,>=latex]
[x=1cm, y=1cm]
\begin{scope}
\draw (0,2) grid [step=1] (9,9);
% \draw [line width=2pt, draw=black, fill=black] (0,1) rectangle (0,2);
% \draw [thick, draw=black, fill=gray, fill opacity=0.4] (0,1) rectangle (2,2);
\draw [ultra thick,->] (5,5) -- (6,7);
\draw [thick,dashed] (0,7.5) -- (9,3);
\draw [thick,dashed] (0,8) -- (9,3.5);
\draw [thick,dashed] (0,8.5) -- (9,4);
\draw [thick,dashed] (0,9) -- (9,9-9/2);
%\filldraw[black]  (6,7) circle (3pt);
\filldraw[black]  (5,5) circle (2pt);
\filldraw[black]  (4,6) circle (4pt);
\filldraw[black]  (5,6) circle (4pt);
\filldraw[black]  (6,6) circle (4pt);
\filldraw[black]  (6,5) circle (4pt);
\node at (4.7,4.7) {\large $\textbf 0$};
\node at (6.2,7.3) {\large $\textbf z$};
\node at (9.5,2.8) {\large $\mathbf \cL_0$};
\node at (9.5,4) {\large $\mathbf \cL_2$};
\node at (9.5,3.3) {\large $\mathbf \cL_1$};
\node at (9.5,4.7) {\large $\mathbf \cL_3$};
\end{scope}
\end{tikzpicture}
  \caption{An example in two dimensions of a constraint satisfying the exterior
    condition w.r.t. a sequence of increasing half-spaces. Only the
    shells $\{\cL_n\}_{n=0}^3$ are drawn. The
    constraint $c_0$ requires that the restriction of the
    configuration $\o$ to each one of the four vertices around the origin
    (black dots) belongs to a certain subset $G\subset S$.}
  \label{fig:exterior}
  \end{figure}
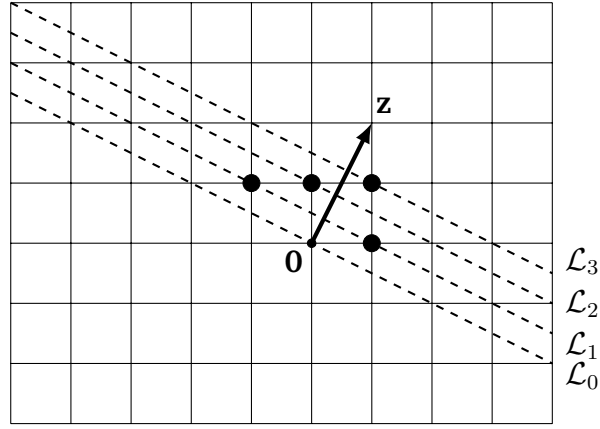
\subsection{Poincar\'e inequality}
\label{sec:main_result}
For simplicity we state our main result directly for 
the infinite lattice $\bbZ^d$. There is also a
finite-volume version in a box $\L\subset \bbZ^d$ which is proved
exactly in the same way. 
In the sequel, given a set $V\subset \bbZ^d$, we will write $\var(f)$ and $\var_V(f)$ for the
variances of a function $f\in L^2(\O,\mu)$ 
w.r.t. to $\mu$ and to $\mu(\cdot \tc \{\o_y\}_{y\notin V})$ respectively.

Let $\{c_x^{(i)}\}_{x\in
  \bbZ^d}$, $i=1,\dots,k,$ be a family of constraints with supports
$\D_x^{(i)}$ and failure probabilities $\e_x^{(i)}$. For any non-empty
$I\subset [k]$ let $\l_I\in (0,+\infty)$ be a positive weight, let $\e^{(I)}_x = \mu(\prod_{i\in I}(1-c^{(i)}_x))$
and let $\D_x^{(I)}=\cup_{i\in I}\D_x^{(i)}$. 
\begin{maintheorem}
\label{thm:1}  Assume that there exists a choice of
$\{\l_I\}_{I\subset [k]}$ such that
\begin{equation}
  \label{eq:key_condition}
\Bigl(\sumtwo{I\subset [k]}{I\neq \emptyset}\l_I\Bigr) \sup_z
\sumtwo{I\subset [k]}{I\neq \emptyset}\ \sumtwo{x\in \bbZ^d}{x\cup \D^{(I)}_x\ni z}\l_I^{-1}\e^{(I)}_x<1/4.
\end{equation}
Suppose in addition that there exists an
exhausting family $\{V_n\}_{n=-\infty}^\infty$ of sets of $\bbZ^d$ such that, for any $i\in
[k],$ the 
constraints $\{c^{(i)}_x\}_{x\in \bbZ^d}$ satisfy the exterior
condition w.r.t. $\{V_n\}_{n=-\infty}^\infty$. Then, for any local (\ie depending on finitely many variables) function $f:\O\mapsto \bbR$,  
 \begin{gather}
\label{CP}
\var(f)\le 4 \sum_x \mu\Bigl(\bigl[\ \prod_{i=1}^kc^{(i)}_x\ \bigr] \var_x(f)\Bigr).
  \end{gather}
\end{maintheorem}
\begin{remark}
\label{rem:PI-gap}
The r.h.s. of \eqref{CP} is the Dirichlet form of a special KCM on
$\bbZ^d$ with constraints $c_x=\prod_{i=1}^kc^{(i)}_x$ (see Section
\ref{sec:models}). Thus \eqref{CP} says that the relaxation time of the above
process (see Definition \ref{def:relax}) is smaller than $4$.      
\end{remark}
\begin{remark}\ 
\label{rem:constraint}
It is easy to construct examples of constraints for which the
exterior condition is violated and the
  r.h.s. of \eqref{CP} is zero for a suitable local function $f$. Take
  for instance $S=\{0,1\}$, $d=2$ and $c_x$ the indicator of the event that at least three
  nearest neighbours of $x$ are in the zero state.
  In this case  
there does not exist an exhausting family $\{V_n\}_{n=1}^\infty$ such
  that the constraints satisfy the exterior condition w.r.t. $\{V_n\}_{n=1}^\infty$.  Furthermore if we let 
 $f(\o)= \o_0\o_{\vec e_1}\o_{\vec e_1+\vec e_2}\o_{\vec e_2}$ then
  $c_x(\o)\var_x(f)=0$ for all $\o$ and all $x\in \bbZ^d$ while $\var(f)>0$. Therefore, for this choice of $c_x$, inequality \eqref{CP} does not hold for all local functions and the KCM with constraint $c_x$ has infinite relaxation time (see Remark \ref{rem:PI-gap}).
 We stress that, however, the fact that the constraints satisfy the exterior condition is not a necessary 
 condition in order for \eqref{CP} to hold. See the following remark for further explanations.
 \end{remark}
\begin{remark}
\label{rem:extension}
For certain applications the following monotonicity property turns out to be useful. Suppose
that $\{c_x^{(i)}\}_{x\in \bbZ^d, \,i\in [k]}$ satisfy the
condition of the theorem and let $\{\hat c_x^{(i)}\}_{x\in \bbZ^d, \,i\in [k]}$ be another
family of constraints which are dominated by the first ones in the sense that
$c_x^{(i)}\le \hat c_x^{(i)}$  for all $i,x$. Then clearly \eqref{CP}
holds for all local functions with $c_x^{(i)}$ replaced bt $\hat c_x^{(i)}$ even if the latter does not
satisfy the exterior condition. As an example take $S=\{0,1\}$, $k=1$ and $\hat c_x$
the constraint that at least two neighbours of $x$ are in the zero state (namely the constraint of FA-2f model, see Section 
\ref{sec:models})
and $c_x$ the same but restricted to the neighbours of the form $x+\vec
e_i$, $i\in [d]$. 
\end{remark}
\begin{proof}[Proof of Theorem \ref{thm:1}]
We first treat the case of a single constraint $k=1$. After that we
will explain how to generalize the argument to $k>1$ constraints. 
We begin with a simple result.
\begin{lemma}For any local function $f$ 
\begin{equation}
\label{AB} \var(f)\leq\sum_x
\mu\bigl(\var_x \bigl(\mu_{\ex}(f) \bigr) 
\bigr).
\end{equation}
\end{lemma}
\begin{proof}[Proof of the Lemma]
Let $\{V_i\}_{i=-\infty}^\infty$ be the exhausting family of sets
w.r.t. which all the constraints satisfy the exterior condition, let
$\cL_i=V_i\setminus V_{i-1}$ be the corresponding $i^{th}$-shell and
assume w.l.o.g. that the support of $f$ is contained in
$\cup_{i=0}^n \cL_i$. Let finally $\L_j=\cup_{i=n-j}^n \cL_i$, $j\le n$. 
% and write $\cF_n$ for
% the $\s$-algebra generated by the variables $\{\eta_x\}_{x\in\L_n}$.
Using the formula for conditional variance together with the
fact that $\mu$ is a product measure we get:
\begin{eqnarray*}
\var(f)&=&\mu\bigl(\var_{\L_0}(f)
\bigr) +\var \bigl(\mu_{\L_0}(f) \bigr)
\\
&=& \mu \bigl(\var_{\L_0}(f) \bigr) +\mu
\bigl(\var_{\L_1} \bigl[\mu_{\L_0}(f) \bigr] \bigr) +
\var \bigl(\mu_{\L_1} \bigl[\mu_{\L_0}(f) \bigr]
\bigr)
\\
&\vdots&
\\
&=&\mu \bigl(\var_{\L_0}[f] \bigr)+\sum
_{j=0}^{n-1}\mu \bigl(
\var_{\L_{j+1}} \bigl[\mu_{\L_{j}}(f) \bigr] \bigr).
\end{eqnarray*}
Recall now the standard inequality valid for any product probability
measure $\nu=\nu_1\otimes\nu_2$:
\[
\var_{\nu}(f)\le \nu(\var_{\nu_1}(f))+\nu(\var_{\nu_2}(f)).
\]
If we apply the inequality to $\var_{\L_{j+1}} \bigl[\mu_{\L_{j}}(f) \bigr]$ and observe that 
$\mu_{\L_j}(f)$ does not depend on the variables in $\L_j$, we get immediately 
\begin{align*}
\mu(\var_{\L_{j+1}} \bigl[\mu_{\L_{j}}(f) \bigr])&\le&
                                                       \sum_{x\in\L_{j+1}\setminus \L_j}
\mu\bigl(\var_x\bigl(\mu_{\L_j}(f) \bigr)\bigr)= \sum_{x\in\L_{j+1}\setminus \L_j}
\mu\bigl(\var_x\bigl(\mu_{\ex}(f) \bigr)\bigr).
\end{align*}
Analogously,
\[
\mu\bigl(\var_{\L_0}[f] \bigr)\le\sum
_{x\in\L_0}\mu \bigl(\var_x(f) \bigr)=\sum
_{x\in\L_0}\mu \bigl(\var_x(\mu_{\ex}(f)) \bigr),
\]
because $\mu_{\ex}(f)=f$ for any $x\in \L_0$.
The proof of the claim is complete.
\end{proof}
We can now prove the theorem for $k=1$ and the starting point is \eqref{AB}.
We begin by examining a generic term $\mu \bigl(\var_x(\mu_{\ex}(f)) \bigr)$ for which we write
\[
\mu_{\ex}(f)=\mu_{\ex}\bigl(c_xf
\bigr) + \mu_{\ex}\bigl(\bigl[1-c_x
\bigr]f\bigr),
\]
so that
\begin{equation}\label{D0}
\var_x \bigl( \mu_{\ex}(f) \bigr) \le2 \var_x
\bigl(\mu_{\ex} \bigl(c_xf \bigr)
\bigr) + 2 \var_x \bigl(\mu_{\ex} \bigl(\bigl[1-
c_x\bigr]f \bigr) \bigr).
\end{equation}
Since $c_x(\o)$ does not depend on $\o_x$, the convexity of the
variance implies that the first term in the above r.h.s. satisfies
\begin{equation*}
\var_x \bigl(\mu_{\ex} \bigl(c_xf
\bigr) \bigr)\le\mu_{\ex} \bigl(\var_x \bigl(
c_xf \bigr) \bigr)= \mu_{\ex} \bigl(c_x\var_x(f) \bigr).
\end{equation*}
We now turn to the analysis of the more complicated second term in the r.h.s. of \eqref{D0}.
\begin{align*}
\var_x \bigl(\mu_{\ex} \bigl(\bigl([1-c_x\bigr]f \bigr) \bigr)&=\var_x \bigl(
\mu_{\ex} \bigl(\bigl[1-c_x\bigr]
\bigl(f-\mu_{ \ex\cup\{x\}}(f)+\mu_{ \ex\cup\{x\}}(f) \bigr) \bigr)\bigr)
\nonumber\\
&=\var_x \bigl(\mu_{\ex} \bigl(\bigl[1-c_x\bigr]g \bigr) \bigr),
\end{align*}
where $g:=f-\mu_{ \ex \cup\{x\}}(f)$ and we used the fact that $\var_x\bigl(\mu_{\ex}([1-c_x]\mu_{\ex \cup\{x\}}(f))\bigr)=0$. 

Recall now that the constraint $c_x$ depends only on $\{\o_y\}_{y\in
  \D_x}$ with $\D_x\subset \ex$. Thus 
\[
\mu_{\ex} \bigl(\bigl[1-c_x\bigr]g \bigr)=
\mu_{\ex}\left([1-c_x]
  \mu_{\ex\setminus \D_x}(g)\right)
\]
and a  Schwarz-inequality
then gives:
\begin{align}
\var_x \bigl(\mu_{\ex}
  \bigl(\bigl[1-c_x\bigr]g \bigr) \bigr)
&\leq
\mu_x \left(
\bigl(\mu_{\ex} \bigl(\bigl(1-c_x\bigr)
\mu_{\ex\setminus\Delta_x }g \bigr) \bigr)^2 \right)
\nonumber\\
&\leq\e_x\mu_{\ex \cup \{x\}} \Bigl(\bigl[
\mu_{\ex\setminus\Delta_x
}(g) \bigr]^2 \Bigr).
\label{least} 
\end{align}
Next we note that
\begin{equation}
\label{last}
\mu_{\ex \cup \{x\}} \Bigl(\bigl[
\mu_{\ex\setminus\Delta_x
}(g) \bigr]^2 \Bigr)=\mu_{x\cup\Delta_x
} \bigl( \mu_{\ex\setminus\Delta_x
} (g)^2
\bigr)=\var_{x\cup\Delta_x }\left( \mu_{\ex\setminus\Delta_x
}(g )\right),
\end{equation}
where we used the fact that $\mu_{x\cup\Delta_x }\left(\mu_{\ex\setminus\Delta_x
}(g )\right)=\mu_{\ex \cup \{x\}}(g)=0$ by the
definition of $g$.
Then by using \eqref{AB}, (\ref{least}) and (\ref{last}) we get
\begin{eqnarray}
\var_x \bigl(\mu_{\ex} \bigl(\bigl[1-c_x\bigr]g \bigr) \bigr)&\leq&\e_x\sum
_{z\in x\cup\Delta_x}\mu_{x\cup
\Delta_x}\left(\var_z\left(
\mu_{Ext_z}\left[ \mu_{ \ex\setminus
\Delta_x }(g)\right]\right)\right)
\nonumber\\
&\leq&\e_x\sum_{z\in x\cup\Delta_x}
\mu_{\ex \cup\{x\}}\left(\var_z(
       \mu_{Ext_z}(g)\right)\nonumber\\
&=&\e_x\sum_{z\in x\cup\Delta_x}
\mu_{\ex \cup\{x\}}\left(\var_z(
       \mu_{Ext_z}(f)\right),
\label{miserve}
\end{eqnarray}
where we use the convexity of the variance to obtain the second inequality.

In conclusion,
\begin{eqnarray}\label{F}\quad
&&
\sum_{x} \mu \left(
\var_x \bigl(\mu_{\ex}(f) \bigr) \right)
\nonumber
\\
&&\le2 \sum_{x}\mu\bigl(c_x\var_x(f) \bigr) + 2\sum
_{x}\e_x\sum_{z\in x\cup\D_x}
\mu\left(\var_z \bigl(\mu_{Ext_z}(f) \bigr)
\right)
\\
&&\le 2\sum_{x}\mu\bigl(c_x\var_x(f) \bigr) + 2\left[\sup_z \sum_{x:\ x\cup \D_x\ni z}\e_x\right]\sum_{z}
\mu\left(\var_z \bigl(\mu_{Ext_z}(f) \bigr)
\right).\nonumber
\end{eqnarray}
If $\sup_z \sum_{x:\ x\cup \D_x\ni z}\e_x\le
1/4$ we get  
\[
\sum_{x} \mu \left(\var_x \bigl(\mu_{\ex}(f) \bigr)
\right)
\le 4 \sum_{x}\mu\bigl(c_x\var_x(f) \bigr).
\]  
We now turn to the general case $k>1$. Let $c_x=\prod_i c_x^{(i)}$ and
recall the definition of $\e_x^{(I)}$ and of $\D_x^{(I)}$ for any
non-empty $I\subset [k]$. Let also $d^{(I)}_x=
\prod_{i\in I}(1-c_x^{(i)})$ so that $\e_x^{(I)}=\mu(d^{(I)}_x)$.
Notice
that (inclusion/exclusion formula)
\[
1-c_x= \sumtwo{I\subset [k]}{I\neq \emptyset}
(-1)^{\text{Parity}(I)+1}d_x^{(I)}=
\sumtwo{I\subset [k]}{I\neq \emptyset}
(-1)^{1+\text{Parity}(I)}\sqrt{\l_I} d_x^{(I)}/\sqrt{\l_I}.
\]
Thus the delicate term $\var_x \bigl(\mu_{\ex} \bigl(\bigl([1-c_x\bigr]f
\bigr) \bigr)$ in \eqref{D0} can be bounded from above using the Schwartz
inequality by
\begin{gather*}
\bigl(\sumtwo{I\subset [k]}{I\neq \emptyset}\l_I\bigr)\  \sumtwo{I\subset [k]}{I\neq \emptyset}
\l_I^{-1} \var_x \bigl(\mu_{\ex} \bigl(d^{(I)}_x\, f
\bigr) \bigr). 
\end{gather*}
At this stage we apply the steps leading to \eqref{miserve} to each
term $\var_x \bigl(\mu_{\ex} \bigl(d^{(I)}_x\, f
\bigr) \bigr)$ to get
\[
\var_x \bigl(\mu_{\ex} \bigl(\bigl([1-c_x\bigr]f
\bigr) \bigr)\le \bigl(\sumtwo{I\subset [k]}{I\neq \emptyset}\l_I\bigr)\sumtwo{I\subset [k]}{I\neq \emptyset}
\l_I^{-1}\e^{(I)}_x\sum_{z\in x\cup\Delta^{(I)}_x}
\mu_{\ex \cup\{x\}}\left(\var_z( \mu_{Ext_z}(f)\right).
\]
As in \eqref{F} we conclude that
\begin{gather*}
\sum_{x} \mu \left(
\var_x \bigl(\mu_{\ex}(f) \bigr) \right)
\\
\le 2\sum_{x}\mu\bigl(c_x\var_x(f) \bigr) + 2 \bigl(\sumtwo{I\subset [k]}{I\neq \emptyset}\l_I\bigr)\Bigl(\ \sup_z 
\sumtwo{I\subset [k]}{I\neq \emptyset}
\ \sumtwo{x}{x\cup \D^{(I)}_x\ni z}\l_I^{-1}\e^{(I)}_x\Bigr)\sum_{z}
\mu\left(\var_z \bigl(\mu_{Ext_z}(f) \bigr)
\right),
\end{gather*}
which proves the theorem if 
\[
\bigl(\sumtwo{I\subset [k]}{I\neq \emptyset}\l_I\bigr)\Bigl(\ \sup_z 
\sumtwo{I\subset [k]}{I\neq \emptyset}
\ \sumtwo{x}{x\cup \D^{(I)}_x\ni z}\l_I^{-1}\e^{(I)}_x\Bigr)\le 1/4.
\]
\end{proof}
\subsection{An application within supercritical
percolation in two dimensions}\label{supercritical}
In this section we restrict ourselves to the case in which the single
site probability space $(S,\hat \mu)$ coincides with
$\left(\{0,1\},B(p)\right)$ and the lattice dimension is equal to
two. Given $\o\in \O:=\O_{\bbZ^2}$ we will say that $x\in \cC(\o):=\{x\in \bbZ^2:\
\o_x=0\}$ belongs to an \emph{infinite
cluster of zeros} if the connected (w.r.t. to the graph
structure of $\bbZ^2$) component of $\cC(\o)$ containing $x$ is
unbounded. It is well known that there exists $p_c\in
(0,1)$
such
that 
\[
\theta(p):= \mu(\text{the origin
  belongs to an infinite cluster})
\] 
is positive iff $p<p_c$ and that
moreover there exists $\mu$-a.s. a unique unbounded component of $\cC(\o)$
%footnote
The conjectured threshold $p_c$ is approximately $1-p_c\approx 0.59$ \cite{Grimmett}. Fix $x\in\mathbb Z^2$, $\omega\in\Omega$ and let  $\bar\omega$ be the configuration obtained from $\omega$ by setting to $1$ site $x$, namely $\bar\omega_x=1$ and $\bar\omega_y=\omega_y$ for $y\neq x$. We let $c^\infty_x(\omega)=1$ if  at least two nearest
neighbors of $x$ belong to an infinite cluster of zeros in the configuration  $\bar\omega$, $c^\infty_x(\omega)=0$ otherwise.  \begin{theorem}
\label{thm:inf-cluster}
There exists $p_0\in (0,p_c)$ such that for any $p\le p_0$ and any local function $f$
\begin{equation}
  \label{eq:thinfinity}
\var(f)\le 4 \sum_x \mu\Bigl(c^\infty_x \var_x(f)\Bigr).
\end{equation}
\end{theorem}
\begin{remark}
It follows in particular that, for all $p$ sufficiently small, the kinetically constrained model
with
constraints $\{c_x^\infty\}_{x\in \bbZ^2}$ (cf. Section
\ref{sec:models}) has its relaxation time bounded by $4$.  
\end{remark}
\begin{proof}
We will make use of the following standard construction for
super-critical percolation \cite{Chayes}. Let
$\ell_n=2^n$ and define 
$R_n$ to be a rectangle of the form either $[\ell_{n}]\times [\ell_{n-1}]$
or $[\ell_{n-1}]\times [\ell_{n}]$ according to whether $n$ is even or
odd. We will also denote by $R^{(1)}_n (R^{(2)}_n)$ the rectangle
obtained by translating $R_n$ by the
vector $-\vec e_1 (-\vec e_2)$ (see Figure \ref{fig:rectangle}).
With the help of the families $\{R^{(1)}_n, R^{(2)}_n\}_{n\in \bbN}$
we finally introduce a new family of constraints as
follows. 

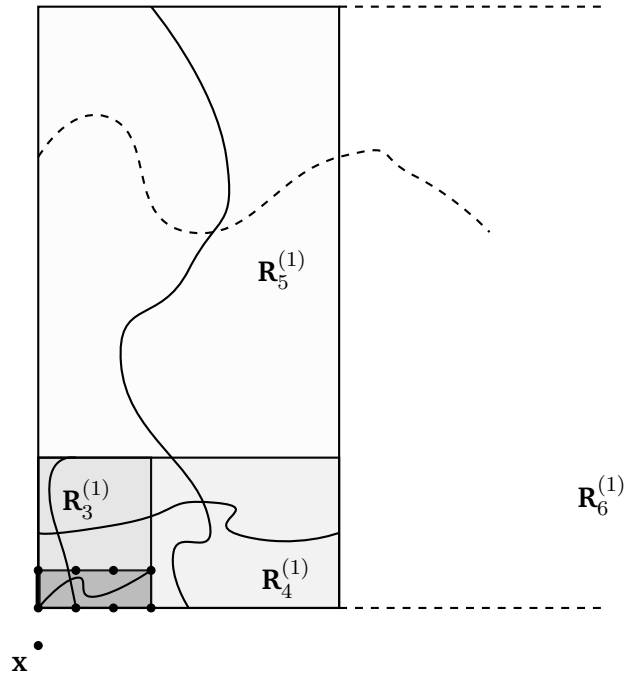
\begin{figure}
\begin{tikzpicture}[scale=0.5,>=latex]
[x=0.5cm, y=0.5cm, semitransparent]
\begin{scope}
%\draw [help lines] (0,0) grid (7,6.5);
%\draw [help lines] (0,0) grid [step=1] (8,17);
\draw [line width=2pt, draw=black, fill=black] (0,1) rectangle (0,2);
\draw [thick, draw=black, fill=gray, fill opacity=0.4] (0,1) rectangle (3,2);
\draw [thick, draw=black, fill=gray, fill opacity=0.1] (0,1) rectangle (3,5);
\draw [thick, draw=black, fill=gray, fill opacity=0.08] (0,1) rectangle (8,5);
\draw [thick, draw=black, fill=gray, fill opacity=0.02] (0,1)
rectangle (8,17);
\draw [thick, dashed] (8,1) -- (15,1);
\draw [thick, dashed] (8,17) -- (15,17);
\filldraw[black]  (0,1) circle (3pt);
%\filldraw[black]  (1,0) circle (3pt);
\filldraw[black]  (0,2) circle (3pt);
\filldraw[black]  (0,0) circle (3pt);
\filldraw[black]  (1,1) circle (3pt);
\filldraw[black]  (2,1) circle (3pt);
\filldraw[black]  (2,2) circle (3pt);
\filldraw[black]  (3,2) circle (3pt);
\filldraw[black]  (3,1) circle (3pt);
\filldraw[black]  (1,2) circle (3pt);
\node at (-0.5,-0.5) {$\textbf x$};
\node at (1.3,3.9) {$\textbf R^{(1)}_3$};
\node at (6.6,1.8) {$\textbf R^{(1)}_4$};
\node at (6.5,10) {$\textbf R^{(1)}_5$};
\node at (15,4) {$\textbf R^{(1)}_6$};
\draw [thick,black] plot [smooth, tension=1] coordinates { (0,1) (1,1.8)
  (1.5,1.3) (3,2)};
\draw [thick,black] plot [smooth, tension=1] coordinates { (1,1) (0.8,2)
  (0.3,4.3) (1,5)};
\draw [thick,black] plot [smooth, tension=1.5] coordinates { (0,3) (2,3.2)
  (4.7,3.8) (5.5, 2.9) (8,3)};
\draw [thick,black] plot [smooth, tension=1] coordinates { (4,1) (3.6,2.5)
(4.5,3.5) (2.2,7.5) (4,10) (5,13) (3,17)};
\draw [thick,dashed, black] plot [smooth, tension=1] coordinates {
  (0,13) (2,14) (4,11)
(8,13) (10,12.5) (12,11) };
\end{scope}
\end{tikzpicture}
  \caption{
A drawing of the first five rectangles $\{R_n^{(1)}+x\}_{n=1}^5$
together with a pictorial
representation of the hard crossings of zeros (the solid lines)
required by the auxiliary
contraint $c^{(n,1)}_x.$ The dashed curved line represents a piece of the
hard crossing for the next rectangle $R_6^{(1)}+x $ (the two
horizontal dashed lines). Notice that each rectangle has its leftmost
lowermost vertex always at $x+\vec e_2$ and that the first rectangle
$R_1^{(1)}$ consists of only two vertices, $x+\vec e_2$ and $x+2\vec e_2$.}
% The union of the black and the dark gray rectangles is  $R_2^{(1)}+x$. The union of the  black, the dark gray and the light gray rectangles is $R_3^{(1)}+x$. 
%   The union of the  black, the dark gray, light gray and white rectangles is $R_4^{(1)}+x$.  The  dots denote the empty sites. Here we depict a configuration for which $c_x^{(0)}=c_x^{(1,1)}=c_x^{(2,1)}=c_x^{(3,1)}=c_x^{(4,1)}=1$. This guarantees the occurrence of a path  of empty sites joining $x+e_1$ to the right short side of $R_4^{(1)}$. }
%  %The empty sites belonging to this path are those
% %depicted in black.}
  \label{fig:rectangle}
  \end{figure}
  Let a path $\g$ in
  $\bbZ^d$ of length $|\gamma|:=k$ be  an ordered sequence of $k$ vertices of
  $\bbZ^2$ such that two consecutive sites are nearest neighbors of
  each other. 
For $i=1,2$ let $c_x^{(n,i)}$ be the indicator function of the
event that inside the rectangle $R_n^{(i)}+x$ there exists an \emph{hard
  crossing}, i.e. a path $\g=(x^{(1)},\dots,x^{(m)})$ joining the two opposite
shortest sides such that $\o_{x^{(j)}}=0$ for all $j\in [m]$.
%\footnote{A path with th properties described in the text is usually referred to as a  }}.
Let also $c^{(0)}_x$ be the indicator of
the event that  $\o_{x+\vec e_1}=\o_{x+\vec e_2}=0$. Notice that, by
construction, the
above constraints satisfy the exterior condition
\ref{ext-cond} w.r.t. to the half-spaces defined in Example
\ref{example:1} with $z=(1,1)$. 
Moreover it is easy to check that
\begin{equation}
  \label{eq:monotone}
c^{(0)}_x\prod_{n=1}^\infty c_x^{(n,1)}c_x^{(n,2)}\le c^\infty_x\quad \forall x,
\end{equation}
so that it is enough to prove the constrained Poincar\'e inequality
\eqref{eq:thinfinity}
with $c^\infty_x$ replaced by $c^{(0)}_x\prod_{n=1}^\infty
c_x^{(n,1)}c_x^{(n,2)}$. More precisely we will prove that, for any $k\in \bbN$ and any
local function $f$,
\begin{equation}
  \label{eq:finN}
\var(f)\le 4 \sum_x \mu\Bigl(c^{(0)}_x\prod_{n=1}^k c_x^{(n,1)}c_x^{(n,2)} \var_x(f)\Bigr).
\end{equation}
The theorem will then follow by taking the limit $k\to +\infty$ and
using \eqref{eq:monotone}. In order to prove \eqref{eq:finN} we want
to apply Theorem \ref{thm:1} which in turn requires finding a family
of weights $\{\l_I\}_{I\subset [k]\cup\{0\}}$ satisfying
\eqref{eq:key_condition}. A standard Peierls argument implies that, for
all $p$ small enough,
\begin{align*}
\mu(1-c_x^{(n,i)})\le e^{-m(p)\ell_{n}},
\end{align*}
with  $\lim_{p\to 0}m(p)=+\infty$.
In particular, recalling the definition of $\e_x^{(I)}$ and
$\D_x^{(I)}$ from Section \ref{sec:main_result}, we have the following bounds: 
\begin{align*}
\e_x^{(I)} &\le  e^{-m(p)\ell_{n(I)}}, &         |\D_x^{(I)}|&\le 3
                                                             \ell^2_{n(I)}
  &
  & \text{ if } n(I):=\max\{i\in I\}>0, \\
\e_x^{(I)} &\le 2p,  &    |\D_x^{(I)}|&\le 2 & &\text{otherwise.}
\end{align*}
Let now $\l_I= e^{-\frac{m(p)}{2}\ell_{n(I)}}$ if $I\neq \{0\}$ and $\l_I=\sqrt{p}$
if $I=\{0\}$. With this choice it is easy to check that there
exists $p_0$ independent of $k$ such that for $p<p_0$
\[
\sum_{I\subset [k]\cup\{0\}}\l_I\le 1/2
\] 
and 
\begin{gather*}
\sup_z
\sumtwo{I\subset [k]\cup\{0\}}{I\neq \emptyset}\ \sumtwo{x\in \bbZ^d}{x\cup
  \D^{(I)}_x\ni z}\l_I^{-1}\e^{(I)}_x 
\le 3
\sumtwo{I\subset [k]\cup\{0\}}{I\neq \emptyset,\ I\neq
  \{0\}}\ell_{n(I)}^2 e^{-\frac{m(p)}{2}\ell_{n(I)}} +4\sqrt{p}\\
\le 3 \sum_{n=1}^\infty 2^n 4^n e^{-\frac{m(p)}{2}\ell_{n}}+4\sqrt{p} \le 1/4.
\end{gather*}
In conclusion \eqref{eq:key_condition} holds for all small enough $p$
independent of $k$ and the theorem follows.
\end{proof}

\section{A general approach to prove a Poincar\'e inequality for kinetically constrained spin models}
\label{application}
In this section we start from the general constrained Poincar\'e inequality
proved in Theorem
\ref{thm:1} to develop a quite robust and
general scheme proving a special kind of Poincar\'e constrained
inequality (cf. Theorem \ref{thm:appl} and Corollary \ref{rhs 9bis})
that will be crucial to determine a sharp upper bound on the mean
infection time of KCM.
% inspired by kinetically constrained
% models. The approach developed below will allow us, in particular, to relate
% the scaling of the persistence time of certain kinetically constrained models near the ergodicity
% threshold to the scaling of the critical length scale of the corresponding bootstrap percolation
% model.
 Concrete
and succesful applications to basic kinetically constrained models
(cf. Theorem \ref{thm:main1}) will be given
in the next section. 

The starting point of our approach is the definition of good and super-good single
site events. 
Given two events $G_1,\, G_2$ in the probability space $(S,\hat
\mu),$ let $p_1:=\hat \mu(G_1)$ and $p_2:=\hat \mu(G_2)$. We will assume that $G_1$ is very likely while $G_2$ is very unlikely. In the sequel we will refer to $G_1$ and $G_2$ as the \emph{good}
and \emph{super-good} events respectively. In the applications, $G_2$
will guarantee the presence of a certain bootstrap critical droplet,
while $G_1$ will guarantee the presence of enough infected vertices to
allow a critical droplet to grow.   
\begin{definition}[Good and super-good paths]
Given $\o\in \O=S^{\bbZ^d}$ we will say that a vertex $x$ is good if
$\o_x\in G_1$ and super-good if $\o_x\in G_2$. We will say that a
path $\g=(x^{(1)},\dots, x^{(k)})$ is a \emph{good path for
  $\o$} if each vertex in $\g$ is good. A path will be called
super-good if it is good and it contains at least one  super-good vertex.
\end{definition}
Before stating the main result we need a last notion. For any mapping  $G_1\overset{\Phi}{\mapsto} G_2$ let 
  \begin{equation}
    \label{eq:10}
\l_\Phi=\max_{\s\in G_2}\sum_{\s'\in G_1:\
  \Phi(\s')=\s}\frac{\hat\mu(\s')}{\hat\mu(\s)},
  \end{equation} 
and, for any $\o$ such that $\o_x\in G_1$, let $\Phi^{(x)}:\O\mapsto
\O$ be given by  
\begin{equation}
 \label{eq:12}
\Phi^{(x)}(\o)_z :=
  \begin{cases}
 \Phi(\o_x)& \text{if $z=x$}\\ 
\o_{z}& \text{otherwise.}
  \end{cases}
\end{equation}
\begin{theorem}
\label{thm:appl}
There exist $\d \ll 1$ and $c>0$ such that, for any $G_1\overset{\Phi}{\mapsto} G_2$ and all $p_1,p_2$ with
$\max(p_2, (1-p_1)\log(1/p_2)^{2})\le \d$, the following holds:  
\begin{gather}
\var(f)
\le c\, ( \l_\Phi p^{-4}_2)^d\  \Bigl[\
\sum_{x}\mu\Bigl(\ \bigl[\prod_{i\in [d]}\id_{\{\o_{x+\vec e_i}\in
      G_2\}}\bigr]\var_x(f)\Bigr) \nonumber  \\
\phantom{-------}+\sum_{x,y:\ d_1(x,y)=1}\mu\Bigl(\id_{\{\o_x\in G_1, \o_y\in G_2\}}\left[f(\Phi^{(x)}(\o))-f(\o)\right]^2\Bigr)\Bigr].
\label{eq:9bis}
\end{gather}
\end{theorem}
\begin{remark}
\label{generalization}
We could have stated Theorem \ref{thm:appl} in a more general form in
which the constraint $\prod_{i\in [d]}\id_{\{\o_{x+\vec e_i}\in
      G_2\}}$, appearing in the first term in the r.h.s. of
    \eqref{eq:9bis}, is replaced by $\prod_{y\in A +x}\id_{\{\o_{y}\in
      G_2\}}$, where $A\subset Ext_0$ is some finite set whose cardinality is
    independent of $p_1,p_2$. For example in two dimensions $A$ could
    be $\{\vec e_1\}\cup \{\vec e_2 + \vec e_1\}\cup\{\vec e_2\}\cup\dots\cup \{\vec
    e_2-m \vec e_1\}$. For future applications \cite{MMT} the freedom given by
    the choice of the set $A$ 
    will be quite crucial. The proof in this slightly more general
    case is identical to the one given below. The same applies for the
    developments discussed in Section \ref{canopaths}.  
\end{remark}
The first term in the r.h.s of \eqref{eq:9bis} is a constrained
Dirichlet form $\cD(f)$ as in the r.h.s. of \eqref{CP}, with constraints $c_x:=\prod_{i\in
  [d]}\id_{\{\o_{x+\vec e_i}\in G_2\}}$. These constraints satisfy
the exterior condition w.r.t. the half-spaces defined in Example
\ref{example:1} with $z=(1,\dots,1)$ but, at the same time, they are very
unlikely (recall that $\hat \mu(G_2)\ll 1$) so that we cannot apply
directly Theorem \ref{thm:1} to our setting. Moreover 
the fact that the $\{c_x\}$ are unlikely implies that a Poincar\'e
inequality of the form
$
 \var(f)
\le C \cD(f)
$
for all local $f$ and some finite constant $C$ cannot hold. To see
that take for instance $\{f_n\}_{n=1}^\infty$ to be a sequence of local functions approximating the indicator of the event that the origin belongs
to an infinite oriented cluster of not super-good vertices
%\footnote{
In other words there exists a infinite path $\g=(x^{(1)},\dots,
  x^{(k)},\dots)$ starting at the origin such that $x^{(i)}\prec x^{(i+1)}$
  and $\o_{x^{(i)}}\notin G_2$ for all $i$. Thus the second term in
the r.h.s. of \eqref{eq:9bis} plays an important role.

Our approach
is first to prove a different kind of constrained Poincar\'e
inequality (cf. Proposition \ref{thm:3}) in which the term in
\eqref{eq:9bis} involving $\Phi$ is missing and the
constraint $c_x$ above
is replaced by the weaker (and very likely) constraint that for all $i\in[d]$ there
exists a super-good path $\g^{(i)}$ in $\bbZ^2\setminus \{x\}$ starting at
$x+\vec e_i$ and of length not larger
than $1/p_2^{2}$. Secondly (cf. Lemma \ref{prop:1}), using repeatedly the mapping $\Phi^{x}$
for each $x\in \g^{(i)}$ starting at the super-good vertex of $\g^{(i)}$, we ``bring'' the super-good vertex of $\g^{(i)}$ at  $x+\vec e_i$. In doing that we pay a cost which is embodied in the second term in the
r.h.s. of \eqref{eq:9bis}.    

\begin{proof}[Proof of Theorem \ref{thm:appl}]
In what follows we assume that we have fixed some mapping  $G_1\overset{\Phi}{\mapsto} G_2$. We begin by proving the first step of the roadmap just described.
\begin{proposition}
\label{thm:3}
There exists $\d \ll 1$ such that, for all $p_1,p_2$ satisfying 
$\max(p_2, (1-p_1)\log(1/p_2)^{2})\le \d$, the following
holds. Let $\id_x$ be the indicator of the event that $\forall i\in [d]$
there exists a super-good path $\g^{(i)}$
of length at most $1/p_2^2$ starting at $x+\vec e_i$. Then, for any local $f$, 
\begin{equation}
  \label{eq:1}
\var(f)\le 4\sum_x \mu\left(\id_x \var_x(f)\right). 
\end{equation}
\end{proposition}
\begin{proof}[Proof of the proposition]
In what follows all the auxiliary constraints that we will need to introduce will satisfy the exterior condition w.r.t. the exhausting
family of half-spaces defined in Example \ref{example:1} with
$z=(1,\dots,1)$.

Let $\ell= \lfloor 2\log(1/p_2)\rfloor,\ L=\lfloor e^\ell\rfloor$ and let us define two families of
constraints $\{c_x^{(1)}, c_x^{(2)}\}_{x\in \bbZ^d}$ as follows: 
\begin{align*}
  c_x^{(1)}&=
  \begin{cases}
    1& \text{if for all $i\in [d]$ and all $k\in [\ell]$ the vertex
      $x+k\vec e_i$ is good,}\\
0& \text{otherwise,}
  \end{cases}\\
c_x^{(2)}&=
  \begin{cases}
    1& \text{if for all $i\in [d]\ \exists$ a super-good path in
      $\ex$ of
      length at most $L$}\\
& \text{ starting in the set $\{x+\vec e_i,\dots , x +\ell
      \vec e_i\}$}\\
0& \text{otherwise.}
  \end{cases}
\end{align*}
Notice that $c_x^{(1)}c_x^{(2)}\le \id_x$. In order to apply theorem \ref{thm:1} to the above constraints we need to verify the key condition \eqref{eq:key_condition}. For
this purpose we begin to observe that the corresponding supports
satisfy $\D_x^{(1)}\subset \cup_{i=1}^d \{x+\vec e_i,\dots , x +\ell
      \vec e_i\}$ and
$\D_x^{(2)}\subset \{y\in \bbZ^d:\ d_1(x,y)\le \ell+L\}$. 
In particular there exists a numerical constant $\hat \d$ such that the condition for the validity of Theorem
\ref{thm:1} holds if 
\begin{equation}
  \label{eq:111}
d\ell \mu(1-c_x^{(1)}) + (\ell +L)^d \bigl(\mu((1-c_x^{(2)}))+
\mu((1-c_x^{(1)})(1-c_x^{(2)})) \bigr)\le \hat \d.
\end{equation}
A simple union bound proves that $
\mu(1-c_x^{(1)})\le d\ell (1-p_1)$,
while standard super-critical percolation bounds 
 valid for large
enough values of $p_1$ prove that
\[
\mu((1-c_x^{(1)})(1-c_x^{(2)}))\le \mu(1-c_x^{(2)})\le d \left( e^{-c\log(1/(1-p_1)) \ell} + (1-p_2)^L\right)
\]
for some constant $c>0$.
%\footnote{ 
Fix e.g. the
  first direction. The probability that none of the vertices $x+\vec
  e_1,\dots, x+\ell\vec e_1$ belong to an infinite good path in
  $\ex$ is exponentially small in $\ell$ while the
  probability that a given path of length $L$ is super-good
  conditionally on being good is at 
least $1-(1-p_2)^L$. 
 It is now immediate to verify that given
$\hat \d>0$ there exists $\d>0$ small enough such
that $\max(p_2, (1-p_1)\log(1/p_2)^{2})\le \d$  implies
\eqref{eq:111}.
\end{proof}
Notice that so far the mapping $\Phi$ played no role. We will now use it in order to bound a generic term
$\mu\left( \id_x \var_x(f)\right)$ appearing in \eqref{eq:1}.   
Without loss of
generality we only treat the case $x=0$. 
\begin{lemma} 
\label{prop:1} In the same setting of Theorem \ref{thm:appl} there
exists $c >0 $ independent of $p_1,p_2,\Phi$ such that
\begin{gather}
  \mu \Bigl(\id_0 \var_0(f)\Bigr) 
\le c \ (\l_\Phi p_2^{-2})^d \ \Bigl[\mu\Bigl(\
  \Bigl[\ \prod_{i\in [d]}\id_{\{\o_{\vec e_i}\in
      G_2\}}\Bigr]\var_0(f)\Bigr) \nonumber \\
\phantom{-------}+\sumtwo{x,y: \in \L\setminus
  \{0\}}{d_1(x,y)=1}\mu\Bigl(\id_{\{\o_x\in G_1, \o_y\in G_2\}}\left[f(\Phi^{(x)}(\o))-f(\o)\right]^2\Bigr)\Bigr],
\label{eq:9}
\end{gather}
where 
$\L$ is the box centered at the origin of side $2\lfloor\, 1/p_2^2\, \rfloor$.  
\end{lemma}
By combining together Lemma \ref{prop:1} and Proposition \ref{thm:3}
we get the statement of the theorem.
\end{proof}
% \begin{remark}
% Notice that the first term in the r.h.s. of \eqref{eq:9bis} is a constrained Poincar\'e inequality with constraints $c_x :=\prod_{i\in [d]}\id_{\{\o_{x+\vec e_i}\in
%       G_2\}}$ (cf. Theorem \ref{thm:1}). As it will be clear from the proof of Proposition \ref{prop:1}, the second term measures instead the const of moving around a super-good vertex. As we will show below (cf. Section \ref{sect:canonical-paths}), for models like the one discussed in the two examples above, both terms can be treated in a very similar way using suitable \emph{canonical-paths in phase space} (cf. e.g. \cite{Saloff}). 
% \end{remark}
\begin{proof}[Proof of Lemma \ref{prop:1}]
Recall that $\id_0$ is the indicator of the event $\cap_{i\in[d]}SG_i$, where  $SG_i$ is the event that 
there exists a super-good path $\g^{(i)}$ in $\bbZ^2\setminus \{0\}$
of length at most $L\equiv
 1/p_2^2$ starting at $\vec e_i$. Clearly $SG_i$ is identical to the event
 that there exists $\g=(x^{(1)},\dots, x^{(L)})\subset
 \bbZ^2\setminus \{0\} $, such that:
 \begin{itemize}
 \item each vertex $x^{(j)}$ appears
 exactly once (\ie the path is simple) and $x^{(1)}= \vec e_i$, 
\item there
 exist $n\le L$ such that $x^{(n)}$ is super-good, 
\item all the vertices $x^{(j)}$ with $j\leq n$ are good.
 \end{itemize}
Fix $i=1$ and let us order in some way the set $\cP$ of simple paths in
$\bbZ^d\setminus \{0\}$ of length $L$ starting at $\vec e_1$. For any
$\o \in \cap_{i\in [d]}SG_i$ let $\g^*$ be the smallest path in $\cP$
satisfying the above set of conditions and let $\nu=\nu(\o)$ be the index of the first super-good vertex in $\g^*$. Thus
\begin{gather}
  \mu \Bigl(\id_0 \var_0(f)\Bigr)=
  \sum_{\g\in \cP}\sum_{n=1}^{L}\mu\Bigl(\id_{\{\g^*=\g\}} \id_{\{\nu=n\}}
    \prod_{j=2}^d\id_{\{SG_{j}\}}\  F\Bigr),
\label{eq:10bis}
\end{gather}
where 
\[
F(\o):= \var_0(f)(\o)=\frac 12 \sum_{\s,\s'\in S}\hat \mu(\s)\hat \mu(\s')\left(f(\o\otimes\s)-f(\o\otimes \s') \right)^2, 
\]
where the notation $\o\otimes\s$ denotes the configuration equal to $\s$ at $x=0$ and equal to 
$\o$ elsewhere.

Given $\g=(x^{(1)},\dots,x^{(L)})\in \cP$ and $n\le L$ together
with $\o\in \prod_{j\in [d]}SG_j$ such that $\g^*(\o)=\g$ and $\nu(\o)=n$, let $\Phi^{(i)}(\o)$ be
given by (recall \eqref{eq:12})
\begin{equation*}
 \Phi^{(i)}(\o)=
  \begin{cases}
 \Phi^{(x^{(i)})}(\o)& \text{if $i\le n-1$}\\ 
\o& \text{if $i=n$}
  \end{cases}
\end{equation*}
Thus the mapping $\Phi^{(i)}$, $i\le n-1$, makes the configuration $\o$
super-good in $x^{(i)}$ and leaves it unchanged elsewhere. For $i=n$ the
mapping $\Phi^{(n)}$ is the identity. With the above notation and using the Cauchy-Schwartz inequality we get
\begin{gather}
F(\o)\le 2F(\Phi^{(1)}(\o))+4 \sum_{\s\in S}\hat \mu(\s)\left( f(\Phi^{(1)}(\o)\otimes\s)-f(\o\otimes \s) \right)^2.
\label{eq:11}
\end{gather}
The first term in the r.h.s. of \eqref{eq:11} gives a contribution to
the r.h.s of \eqref{eq:10bis} not larger than
\begin{align}
2 \l_\Phi \mu\Bigl(\id_{\{\o_{\vec e_1}\in G_2\}}\prod_{j=2}^d\id_{\{SG_{j}\}}\var_0(f)\Bigr).
\label{eq:12tris}
\end{align}
Above, after the change of variable $\eta:=\Phi^{(1)}(\o)$, we used \eqref{eq:10} together with the obvious facts that
$\eta$ is super-good at $\vec e_1$ and it belongs to
$\prod_{j=2}^d\id_{SG_{j}}$. 

In order to bound from above the contribution of the second term in the
r.h.s. of \eqref{eq:11} we write
\begin{align}
 \left(f(\Phi^{(1)}(\o)\otimes\s)-f(\o\otimes \s)\right)^2
  &=\left(\sum_{i=1}^{n-1} \left[f(\Phi^{(i+1)}(\o)\otimes\s)-f(\Phi^{(i)}(\o)\otimes
  \s)\right]\right)^2 \nonumber \\
&\le (n-1)\sum_{i=1}^{n-1} \left[f(\Phi^{(i+1)}(\o)\otimes\s)-f(\Phi^{(i)}(\o)\otimes
  \s)\right]^2\nonumber \\
& \le L\sum_{i=1}^{n-1} \left[f(\Phi^{(i+1)}(\o)\otimes\s)-f(\Phi^{(i)}(\o)\otimes
  \s)\right]^2.
\label{eq:12bis}
\end{align}
In turn each summand is bounded from above by
\begin{gather*}
2 \left[f(\Phi^{(i+1)}(\Phi^{(i)}(\o))\otimes
  \s)-f(\Phi^{(i)}(\o)\otimes
  \s)\right]^2 + 2\left[f(\Phi^{(i+1)}(\o)\otimes\s)-f(\Phi^{(i+1)}(\Phi^{(i)}(\o))\otimes
  \s)\right]^2.  
\end{gather*}
Using the fact that
$\Phi^{(i+1)}(\Phi^{(i)}(\o))=\Phi^{(i)}(\Phi^{(i+1)}(\o))$, we see
that both terms in the r.h.s. above have a similar structure. We will
therefore treat explicitly only the first one. Recalling that $\L$ is the box centered at the origin with side $2\lfloor\, 1/p_2^2\, \rfloor$, we get
\begin{gather*}
 2L\mu\Big(\sum_{\g\in
   \cP}\sum_{n=1}^{L}\sum_{i=1}^{n-1}\id_{\{\g^*=\g\}}
 \id_{\{\nu=n\}}
\prod_{j=2}^d\id_{\{SG_j\}}\times \\\times \sum_{\s\in  S}\hat\mu(\s) \left[f(\Phi^{(i+1)}(\Phi^{(i)}(\o))\otimes
  \s)-f(\Phi^{(i)}(\o)\otimes
  \s)\right]^2\Bigr)\\
= 
2L \mu\Bigr(\id_{SG_1}\prod_{j=2}^d\id_{\{SG_j\}}\sum_{i=1}^{\nu-1}\left[f(\Phi^{(i+1)}(\Phi^{(i)}(\o))\otimes
  \s)-f(\Phi^{(i)}(\o)\otimes
  \s)\right]^2\Bigr)
\\
\le 2L \sumtwo{x,y\in \L\setminus
  \{0\}}{d_1(x,y)=1}\mu\Bigl(\id_{\{\o_x\in G_1,\ \o_y\in G_1\}}\left[f(\Phi^{(x)}(\Phi^{(y)}(\o))\otimes
  \s)-f(\Phi^{(y)}(\o)\otimes
  \s)\right]^2\Bigr).
\end{gather*}
After the change of variable $\eta\equiv \Phi^{(y)}(\o)$ inside the
expectation, the above quantity can be bounded
from above by 
\[
2L \l_\Phi \sumtwo{x,y\in \L\setminus
  \{0\}}{d_1(x,y)=1}\mu\Bigr(\id_{\{\eta_x\in G_1, \eta_y\in G_2\}}\left[f(\Phi^{(x)}(\eta))\otimes
  \s)-f(\eta\otimes
  \s)\right]^2\Bigr). 
\]
Putting all together we get that there exist a constant $c>0$ such
that
\begin{gather*}
  \mu \Bigl(\id_0 \var_0(f)\Bigr) 
\le c \l_\Phi p_2^{-2} \Bigl[\mu\Bigl(\id_{\{\o_{\vec e_1}\in
      G_2\}}\prod_{j=2}^d\id_{\{SG_{j}\}}\var_0(f)\Bigr) \\
+\sumtwo{x,y\in \L\setminus
  \{0\}}{d_1(x,y)=1}\mu\Bigl(\id_{\{\eta_x\in G_1, \eta_y\in G_2\}}\left[f(\Phi^{(x)}(\eta))\otimes
  \s)-f(\eta\otimes
  \s)\right]^2\Bigr)\Bigr].
\end{gather*}
We can now analyse the first term inside the above square bracket by
repeating the above analysis for the second direction. In $d-1$ steps
the proof is complete.
\end{proof}
\subsection{A canonical-paths bound on the r.h.s. of
  (\ref{eq:9bis}).}
\label{canopaths}
In this section we proceed by analysing the r.h.s. of
\eqref{eq:9bis} in the special case in which $S=\{0,1\}^V$,
$V=\prod_{i=1}^d[n_i]$ for some integers $\{n_i\}_{i=1}^d$,
and $\hat \mu$ is the Bernoulli(p) product measure. We will write
$|V|$ for the cardinality of $V$. In this setting
the probability space $(S^{\bbZ^d},\mu)$ becomes isomorphic to
$(\O,\mu)$ where $\O=\{0,1\}^{\bbZ^d}$ and $\mu$ is the
Bernoulli(p) product measure. It is therefore convenient to do a relabelling
of the variables $\o\in S^{\bbZ^d}$ as
follows. 

Let $\bbZ^d(\vec n)$ be the renormalised lattice
$\otimes_{i=1}^d (n_i\bbZ)$ and let, for $x\in \bbZ^d(\vec n)$,
$V_x:=V+x$. We will 
write $x\sim y$ iff $x,y$ are nearest neighbors in the renormalised
lattice $\bbZ^d(\vec n)$.
The old
``block'' variable $\o_x\in S$ associated to $V_x$ is renamed as $\o_{V_x}=\{\o_y\}_{y\in
  V_x}$  with now $\o_y\in \{0,1\}$ for all $y$'s. In particular the local variance term $\var_x(f)$ appearing in the r.h.s. of \eqref{eq:9bis} becomes $\var_{V_x}(f)$. Accordingly we rewrite the mapping
$\Phi^{(x)}$, $x\in \bbZ^d(\vec n)$, as $\Phi^{(V_x)}$. 

In order to formulate our bounds we need to define the \emph{canonical
  paths} (cf. e.g. \cite{Saloff}).
\begin{definition}[canonical-paths]
Let $\o,\o'\in \O$  be two configurations which differ in finitely many
vertices. We say that $\G_{\o,\o'}\equiv (\o^{(1)},\o^{(2)},\dots, \o^{(k)}) $ is a
canonical-path between $\o,\o'$  if (i) $\o^{(1)} =\o,\ \o^{(k)}=\o'$,
(ii) $\o^{(i)}\neq \o^{(j)}$ for all $i\neq j$ (no loops) and (iii) for
any $i\in [k-1]$ the configuration $\o^{(i+1)}$ is
  obtained from $\o^{(i)}$ by a single spin flip. The integer $k$ will
  be referred to as the length of the path. 
\end{definition}
The bounds on the individual terms in the r.h.s. of \eqref{eq:9bis} are then as
follows.
\begin{lemma}
\label{A}
We assume that, for any $x\in \bbZ^d(\vec n)$, any $z\in V_x$ and any $\o$ such that $\o_{V_x+\vec
  e_i}\in G_2$ for all $i\in [d]\,$, a canonical-path
$\G_{\o,\o^z}$ has been defined such that a generic transition in the path consists of
a spin flip in $V_x\cup (\cup_{i=1}^d \{V_x+\vec e_i\})$. Let 
\[
\rho_A = \sup_{x\in \bbZ^d(\vec n)}\max_{z\in V_x}\sup_{\o'} \sumtwo{\o:\ \o_{V_x+\vec e_i}\in
  G_2,\,\forall i\in [d]}{\o'\in \G_{\o,\o^z}} \frac{\mu(\o)}{\mu(\o')}.
\]
be the \emph{congestion constant} of the family of canonical-paths and let $N_A$ be their maximal length.
Then 
\[
\sum_{x\in \bbZ^d(\vec n)}\mu\Bigl(\Bigl[\ \prod_{i\in [d]}\id_{\{\o_{V_x+\vec e_i}\in
      G_2\}}\Bigr]\var_{V_x}(f)\Bigr) 
\le c \rho_A N_A |V|^2\sum_{y\in \bbZ^d} \mu\bigl(\id^A_{y}(\o)\var_y(f)\bigr), 
\]
for a numerical constant $c>0$, where $\id^A_y(\o)$ is the indicator of the event that there
exists $x\in \bbZ^d(\vec n)$, $z\in V_x$ and $\bar \o$ such that $\bar
\o_{V_x+\vec e_i}\in
  G_2,\,\forall i\in [d]$ and the pair
$(\o,\o^y)$ form a transition of the canonical-path
between $\bar \o$ and $\bar \o^{z}$.
\end{lemma}
   \begin{lemma}
 \label{B}
We assume that, for any $x\sim y$ and any $\o\in \O$ such that
 $\o_{V_x}\in G_1$ and $\o_{V_y}\in G_2$, a canonical-path
between $\o$ and $\Phi^{(V_x)}(\o)$ has been defined such that a generic transition in the path consists of
a spin flip in $V_x\cup V_y$. Let
  \begin{equation*}
\rho_B= \sup_{\o'}\sup_{x\sim y}\ \sumtwo{\o:\ \o_{V_x}\in G_1,\, \o_{V_y}\in
  G_2}{\o'\in \G_{\o,\Phi^{(V_x)}(\o)}} \frac{\mu(\o)}{\mu(\o')}    
  \end{equation*} 
and let $N_B$ be the maximal length of the paths. 
Then
\begin{gather*}
\sum_{x\sim y}\mu\Bigl(\id_{\{\o_{V_x}\in G_1, \o_{V_y}\in
  G_2\}}\left[f(\Phi^{(V_x)}(\o))-f(\o)\right]^2\Bigr)\Bigr]\\
\le c\rho_B N_B \,|V|\sum_{z\in \bbZ^d} \mu\Bigl(\id^B_z(\o)\var_z(f)\Bigr)
\end{gather*}
for a numerical constant $c>0$, where $\id^B_z(\o)$ is the indicator of the event that there
exists $x\sim y$ and $\o'$ such that $\o'_{V_x}\in G_1,\ \o'_{V_y}\in
G_2$ and the pair
$(\o,\o^z)$ form a transition of the canonical-path between $\o'$
and $\Phi^{(V_x)}(\o')$.
  \end{lemma}
The proof of the above two lemmas is practically identical so we only
prove the first one.  
\begin{proof}[Proof of Lemma \ref{A}]
The starting inequality is 
\[
\var_{V_x}(f)\le\sum_{z\in V_x}\mu_{V_x}(\var_z(f)).
\]   
For simplicity in the sequel we assume $x=0$. Given $\o$ such that $\o_{V+\vec e_i}\in
  G_2\,\forall i\in [d]$ and $z\in V$, let $\G_{\o,\o^z}=(\o^{(1)},\o^{(2)},\dots, \o^{(k)})$ be
  the corresponding canonical-path. Then 
\[
\var_z(f)(\o)= p(1-p)[f(\o^z)-f(\o)]^2\le  p(1-p)k \sum_{j=1}^k [f(\o^{(i+1)})-f(\o^{(i)})]^2,
\]
so that
\begin{gather*}
  \mu\Bigl(\id_{\{\o_{V+\vec e_i}\in
  G_2\,\forall i\in [d]\}}\var_z(f)\Bigr) 
\le N_A p(1-p)\mu\Bigl(\sum_{i=1}^{k-1}\left[f(\o^{(i+1)})-f(\o^{(i)})\right]^2\Bigr)
\\
\le c\rho_A N_A\sum_{y\in V\cup
  (\cup_{i=1}^dV+\vec e_i)}\mu\Bigl(\id^A_y(\o)\var_y(f)\Bigr),
\end{gather*}
 where $\id^A_y(\o)$ is as in the statement and, after the change of
 variables $\o=\o^{(i)}$, we used the definition of
 $\rho_A$ to bound the relative density between $\o^{(i)}$
 and $\o$. The statement of the lemma now follows at once.
\end{proof}
For future purpose we summarise the conclusion of our bounds.  
\begin{corollary}
    \label{rhs 9bis}
Under the same assumptions as in Lemmas \ref{A} and \ref{B}
\begin{gather*}
\var(f)\le c\, ( \l_\Phi p^{-4}_2)^d\ \Bigl[\rho_AN_A|V|^2 \sum_z \mu\Bigl(\id^A_z(\o)\var_z(f)\Bigr)\\  + \rho_BN_B|V| \sum_z \mu\Bigl(\id^B_z(\o)\var_z(f)\Bigr) \Bigr]
\end{gather*}
  \end{corollary}
  \begin{remark}
\label{rem:PI}
In the application to KCM the choice of the canonical-paths entering in the above corollary will always be such that  
$\max\left(\id^A_z(\o),\id^B_z(\o)\right)\le c_z(\o),$
where $c_z$ is the constraint of the KCM at $z\in \bbZ^d$.
Thus in
this case the conclusion of the Corollary implies a Poincar\'e inequality $\var(f)\le C\cD(f)$, where $\cD(f)=\sum_z\mu(c_z\var_z(f))$ is the Dirichlet form of the KCM (cf. Remark \ref{rem:PI-gap}) and $C$ satisfies
\[
C\le c\, ( \l_\Phi p^{-4}_2)^d\Bigl(\rho_AN_A|V|^2 +\rho_BN_B|V|\Bigr).
\]   
  \end{remark}
% Given $x\sim y$ and $\o$ such that $\o_{V_x}\in G_1$ and $\o_{V_y}\in
% G_2$, let $\G=(\o^{(1)},\o^{(2)},\dots, \o^{(k)})$ be the associated
% canonical-path. Then
% \[
% \left[f(\Phi^{(V_x)}(\o))-f(\o)\right]^2\le k
% \sum_{i=1}^{k-1}\left[f(\o^{(i+1)})-f(\o^{(i)})\right]^2,
% \]       
% so that
%   \end{proof}
% It is convenient to first write
% \[
% \var_{V}(f)\le\sum_{z\in V}\mu(\var_z(f)),
% \]   
% and then assume, analogously to Assumption \ref{cpaths}, that, for any $z\in V$ and any $\o$ such that $\o_{V+\vec e_i}\in G_2$ for all $i\in [d]$, there exists a canonical-path 
% $\G_\o^z =(\o^{(1)},\o^{(2)},\dots, \o^{(k)})$, of length $k\le 30\bigl(\prod_{i=1}^d n_i\bigr)$ and such that a generic transition in $\G^z_\o$ consists of a spin flip in $V\cup (\cup_{i=1}^d \{V+\vec e_i\})$, which connects $\o$ to $\o^z$. 
\section{Application to specific KCM models}
\label{sec:models}
 In this section we begin by recalling the definition of the
 Fredrickson-Andersen constrained
spin models with $k$-facilitation (FA-kf in the sequel) introduced by H.C. Andersen and
G.H. Fredrikson in \cite{FH} and of the GG-KCM. As
it will be clear in a moment, 
the FA-kf models are 
closely related to the so-called $k$-neighbour model in bootstrap
percolation, while the GG-KCM model is related to the anisotropic
bootstrap percolation model introduced by Gravner-Griffeath \cites{GG2,GG}. As such, the dynamical properties of both models near the
ergodicity threshold are intimately
related to the scaling properties of the corresponding bootstrap percolation models
in the same regime. Finally we state our main result relating the
persistence time with the critical bootstrap percolation length. 
This will be proven in section 5 using Corollary \ref{rhs 9bis}. The key step will consist in finding suitable (\ie depending on the specific choice of the constraints) good and super-good events $G_1,G_2$, map $\phi$ and canonical-paths.

\subsubsection{The models}\label{sec:models} We will work with the probability space
$(\O,\mu)$ where $\O=\{0,1\}^{\bbZ^d}$ and $\mu$ is the product
Bernoulli($p$) and we will be interested in the asymptotic regime
$q\downarrow 0$ where $q=1-p$. A generic \emph{kinetically
  constrained model} (KCM in the sequel) is a particular interacting particle
system, \ie a Markov process on $\O$, described by the Markov
generator 
\[
(\cL f)(\o)= \sum_{x\in \bbZ^d}c_x(\o)\bigl(\mu_x(f)-f\bigr)(\o),
\] 
where $\mu_x(f)$ is the Bernoulli($p$)-average of $f(\o)$ w.r.t. to
the variable $\o_x$. The constraints $\{c_x\}_{x\in \bbZ^d}$ are  
defined as follows. Let $\cU=\{U_1,\dots,U_m\}$ be a finite collection of finite subsets
of $\bbZ^d\setminus\{0\}$. We call $\cU$ the \emph{update family} of the process and each $X\in \cU$ an \emph{update rule}. Then $c_x$ is the indicator function of the
event that there exists an update rule $X\in \cU$ such that $\o_y=0\
\forall y\in X +x$. We emphasise that we do not assume that the
constraints satisfy the exterior property of Section \ref{sec:constraints}. 
Using these assumptions it is easy to check (cf. \cite{CMRT} for a
detailed analysis) that $\cL$ becomes the generator of a 
reversible Markov process on $\O$, with reversible measure $\mu$. 

In the FA-kf model one takes as $\cU$ the family of $k$-subsets of the
set of nearest neighbors of the origin. In the GG-KCM model in two
dimensions one takes 
$\cU$ as the family of $3$-subsets of the
set of nearest neighbors of the origin together with the vertices $\{\pm
2\vec e_1\}$. In the terminology of bootstrap percolation
(see e.g. \cite{BD-CMS}  and the recent survey \cite{Robsurvey})
the  update family of FA-kf for $k\in[2,d]$ belongs to the class of \emph{critical balanced models}
and the update family of GG-KCM is \emph{critical unbalanced}. Such a difference will appear clearly in the sequel.  
\begin{remark}
Given a KCM with update family $\cU$, we will sometimes refer to the
corresponding bootstrap percolation process as the monotone process
defined in \eqref{eq:2} using the same update rules of the KCM.
\end{remark}
We now define the two main quantities characterising a
KCM. 
\begin{definition}
\label{def:relax}The relaxation time $\trel(q;\cU)$ of the generator
$\cL$ is the best constant $C$ in the Poincar\'e inequality 
\begin{equation}
  \label{eq:gap}
\var(f)\le C\cD(f) \quad \text{ for all local} \ f,
\end{equation}
where $\cD(f)=\sum_x \mu\bigl(c_x \var_x(f)\bigr)$ is the Dirichlet form associated to $\cL$.
\end{definition}
A finite relaxation time implies  (see e.g. \cite{Liggett1}) that the reversible measure $\mu$ is mixing for the
semigroup $P_t$ with exponentially decaying time auto-correlations,
 $$\var\left(e^{t\cL} f\right)\leq e^{-t/\trel}\var(f),\qquad f\in L^2(\mu).$$
The second (random) quantity is the first time the spin at the origin reaches the zero
state:  
\[
\t_0=\inf\{t\ge 0:\ \o_0(t)=0\}.
\] 
In the physics literature the hitting time $\t_0$ is usually referred
to as the \emph{persistence time}, while, in the bootstrap
percolation framework, it would be more conveniently dubbed  
\emph{infection time}. 
In \cite{Praga}*{Theorem 4.7} it was proved that 
\[
\bbP_\mu(\t_0\ge t) \le \exp\bigl(-q\, t/\trel(q;\cU)\bigr),
\] 
implying 
\[
\bbE_\mu(\t_0)\le \trel(q;\cU)/q.
\] 
A matching lower bound in terms of $\trel(q;\cU)$
is missing and we have instead the following result whose proof is
deferred to the appendix. Recall that $\t_{BP}$ is the infection time
of the origin for the corresponding bootstrap percolation process. 
\begin{lemma}
\label{LEM:PARIS}
There exists
$\d=\d(\cU)\in (0,1)$ such that, for all $q$ small enough,
\begin{equation*}
\bbE_\mu(\t_0)\ge \d \mu(\t_{BP})\ge \frac{\d}{2} T_c(q;\cU).
\end{equation*}  
\end{lemma}
One of the main results of \cite{CMRT} states that all the KCM with update family
$\cU$ such that $q_c(\cU)=0$ have a finite relaxation time $\trel$ and
thus a finite mean infection time $\bbE_\mu(\t_0)$.
In particular the above
holds for the FA-kf for $k\in[2,d]$ model and the GG-KCM %  hence a
% finite mean infection time of the origin.$\bbE_\mu(\t_0)$ are
% finite for any $q>0$, where $\bbE_\mu(\cdot)$ denotes the
% average w.r.t. the law of the stationary KCM. The methods of \cite{CMRT} together with
% the results of \cite{BD-CMS} also prove this result for the GG-KCM. 
and our main aim is to compute the rate at
which $\trel$ and $\bbE_\mu(\t_0)$ diverge as $q\to 0$ in both cases. % In order to compare our results to similar
% divergences found in the bootstrap percolation version
% of the two models, we first need the following definitions (see \cite{BSU}).
% \begin{definition}[Bootstrap process on $\bbZ^d_n$]
% \label{BP}
% Given an update family $\cU$, a set $A\subset \bbZ^d_n$ and $\o\in \{0,1\}^{\bbZ^d_n}$ such that $\o_x=0$ iff $x\in A$, one
% sets recursively for $t\in \bbN$, 
% \[
% A_{t+1}=A_t \cup \{x\in \bbZ^d_n:\ x+U_k\subset A_t\ \text{ for some }
% k\in [m]\},\quad A_0=A.
% \]
% We then define the $\cU$-\emph{update closure} of $A$ the set 
% \[
% [A]_\cU=\cup_{t=0}^\infty A_t.
% \]  
% \end{definition}
% \begin{definition}
% \label{critical length}
% We say that $A\subset \bbZ^d_n$ is $q$-random and we will write $\bbP_q$
% for its law, if $A$ coincides with the
% set $\{x\in \bbZ^d_n:\ \o_x=0\}$, $\o\sim \mu$. We then define the \emph{critical
% probability} $q_c(n;\cU)$ and the \emph{critical length} $L_c(q;\cU)$
% of the update family $\cU$ as 
% \begin{align*}
% q_c(n;\cU) &=\inf\{q:\   \bbP_q([A]_\cU=\bbZ^d_n)\ge 1/2\},\\
% L_c(q;\cU)&=\min\{n:\ q_c(n,\cU)\le q\}.
% \end{align*}
% \end{definition}
\subsection{Main result}
We begin to recall what is known on 
the asymptotic scaling as $q\to 0$ of the critical length $L_c(q;\cU)$ defined in
\eqref{eq:3} and the relaxation time
$\trel(q;\cU)$, when the update family $\cU$ is that of the
FA-kf and the GG-KCM models. 

For the update family $\cU$ of the FA-kf model on $\bbZ^d,$ it was proved in \cite{BBD-CM} (see the
introduction there for a short account of previous relevant results)
that for any $d\ge k\ge 2$ there
exists an explicit constant $\l(d,k)$ such that
\begin{equation}
  \label{eq:Lc}
L_c(q;k,d)\equiv L_c(q;\cU)=\exp_{(k-1)}\Bigl(\frac{\l(d,k)+o(1)}{q^{1/(d-k+1)}}\Bigr),
\end{equation}
where $\exp_{(r)}$ denotes the $r$-times iterated exponential,
$\exp_{(r+1)}(x)=\exp(\exp_{(r)}(x))$.
For the case of the GG-KCM, it was established in \cite{DC-Enter} (see also \cite{DPEH}
for a detailed analysis of the $o(1)$ term below)
that instead 
    \[L_c(q;\cU)=\exp\Bigl(\frac{(\log(1/q))^2}{12
      q}(1+o(1))\Bigr).\]
As far as the asymptotic
behaviour $\trel(q;\cU)$ as $q\to 0$ is concerned, only the FA-kf model has been
considered so far and the following bounds have been proved in \cite{CMRT}. There
exists $c>0$ such that
\begin{align*}
  L_c(q;\cU)^{1-o(1)}\le &\trel(q;\cU)\le \exp\bigl(c/q^5\bigr)
                    \hskip 1cm &d=k=2,\\
  L_c(q;\cU)^{1-o(1)}\le &\trel(q;\cU)\le \exp_{(d-1)}\bigl(c/q\bigr)
                    \hskip 1cm &d\ge 3, \ k\le d.
\end{align*}
Notice that the above upperbounds are very far from $L_c(q;\cU)$.  % In \cite{CMRT}*{Theorem 3.6} it was also proved that the large
% deviations of  $\t_0$ can be controlled in terms of  
% $\trel(q;\cU)$. More precisely it holds in great
% generality that
% \[
% \bbP_\mu(\t_0\ge t) \le \exp\bigl(-c q\, t/\trel(q;\cU)\bigr)
% \] 
% for some $c>0$ independent of $q$. In particular $\bbE_\mu(\t_0)= 
% O(\trel(q;\cU)/q)$. A matching lower bound in terms of $\trel(q;\cU)$
% is missing and we have instead the following result whose proof is
% deferred to the appendix. Recall that $\t_{BP}$ is the infection time
% of the origin for the bootstrap process. 
% \begin{lemma}
% \label{LEM:PARIS}
% There exists
% $\d=\d(\cU)\in (0,1)$ such that, for all $q$ small enough,
% \begin{equation*}
% \bbE_\mu(\t_0)\ge \d \mu(\t_{BP}).
% \end{equation*}  
% \end{lemma}
In conclusion, while the control of the critical length $L_c(q;\cU)$
is rather sharp, the relaxation time $\trel(q;\cU)$ and the mean
hitting time
$\bbE_\mu(\t_0)$ are still poorly controlled. The main outcome of the
theorem below is a much tighter connection between $\trel(q;\cU)$, and
therefore $\bbE_\mu(\t_0)$, and $L_c(q;\cU)$.
\begin{theorem}
 \label{thm:main1}
For the FA-2f model in $\bbZ^d$ and the GG-KCM  model there
exists $\a>0$ such that  
\begin{equation}
  \label{main1:1}
\trel(q;\cU)=O\Bigl(L_c(q;\cU)^{\log(1/q)^\a}\Bigr).
\end{equation}
For FA-kf model in $\bbZ^d$ with $3\le k\le d$ there exists $c>\l(d,k)$ such that
\begin{equation}
  \label{main1:2}
\trel(q;\cU) \le \exp_{(k-1)}\bigl(c/q^{1/(d-k+1)}\bigr).
\end{equation}
\end{theorem}
\section{Proof of theorem \ref{thm:main1}}
\label{proofmainthm}
\subsubsection{Reader's guide  and notation} The proof of the theorem uses all the machinery which was developed in
the previous sections. Therefore, for all the above models, the coarse-grained
probability space $(S,\hat \mu)$ (cf. e.g. the beginning of Section
\ref{canopaths}) will be of the form $S=\{0,1\}^{V}$, with
$V=\prod_{i=1}^d[n_i]$ and $\hat \mu$ the product Bernoulli($p$) measure. 

The starting point of the proof is to make an appropriate choice for the value
of $\vec n=(n_1,\dots,n_d)$ together with a working definition of the good and super-good
events $G_1,G_2\subset S$ and of the mapping
$G_1\overset{\Phi}{\mapsto} G_2$ (cf. Section \ref{application}) for
each model. Clearly, in order to apply Theorem \ref{thm:appl} and 
Corollary \ref{rhs 9bis}, our choice of $(\vec n,G_1,G_2)$ must ensure that
the probabilities $p_1=\hat\mu(G_1)$ and $p_2=\hat \mu(G_2)$ satisfy the
basic condition $\lim_{q\to
  0}(1-p_1)\bigl(\log(1/p_2)\bigr)^2=0$ of Theorem \ref{thm:appl}. In the
FA-kf models no direction plays a special role (it is a balanced model
in the language of \cite{Robsurvey}) and therefore we choose $n_i=n$ for all $i\in
[d]$. In the GG-KCM  the above symmetry is broken and we will need
to distinguish between the two directions. This part of the proof is carried out in Part I (see below). The second part of the proof (cf. Part II below) involves defining appropriately the canonical-paths appearing in Lemma \ref{A} and \ref{B} (see also Corollary \ref{rhs 9bis}) and bounding the corresponding length and congestion constants.  

Carrying out the above program could become particularly heavy from a notational point of view. Therefore we will sometimes adopt a more descriptive and informal approach. More specifically, given a configuration $\o\in\{0,1\}^{\bbZ^d}$ and a region $\L\subset \bbZ^d$, we will declare $\L$ \emph{empty (occupied)} if $\o\restriction\L=0\ (1)$. While constructing the canonical-paths appearing in Lemmas \ref{A} and \ref{B} we will say that we \emph{empty (fill)} $\L$ if we flip to $0\, (1)$, one by one according to some preassigned schedule (\ie an ordering of the to-do flips), all the occupied/empty sites of $\L$. It is important to emphasise that the schedules involved in the operations of emptying or filling a region will \emph{always} be such that each spin flip dictated by the schedule will occur while fulfilling the specific constraint of each model. Schedules with this property will be dubbed \emph{legal schedules}. 
A closely related notion is that of \emph{legal canonical-path}.
\begin{definition}
Given a KCM, let $\{c_x\}_{x\in \bbZ^d}$ be the corresponding family of
constraints. A \emph{legal canonical-path} between two configurations
$\o,\o'$ is a canonical-path $\G_{\o,\o'}\equiv
(\o^{(1)},\o^{(2)},\dots, \o^{(m)}) $ with the additional property that
$c_{x^{(i)}}(\omega^{(i)})=1\ \forall i\in [m-1]$, where $\o^x$
denotes the configuration obtained from $\o$ by flipping the value
$\o_x$ and $x^{(i)}$ is the vertex such that $\omega^{(i+1)}=(\omega^{(i)})^{x^{(i)}}$. We say that the canonical-path is \emph{decreasing (increasing)} if for any $i\in [m-1]$ and any $x\in \bbZ^d$ $\o^{(i+1)}_x\le \o^{(i)}_x$ ($\o^{(i+1)}_x\ge \o^{(i)}_x$).   
\end{definition}
Next, we recall the notion of a subset of $\bbZ^d$ being
\emph{internally spanned} which will play a crucial role in the
definition of the good and super-good events for the specific KCM
treated here. 
\begin{definition}[Internally spanned]
Consider a KCM with updating family $\cU$. Given $\L\subset \bbZ^d,$
we will denote by 
$I(\cU,\L)\subset \O_\L$ the event that $\L$ is \emph{$\cU$-internally
spanned}, \ie that
\[
\left[\{x\in \L:\ \o_x=0\}\right]_{\cU}=\L.
\]
When the update family
is that of the FA-kf
model in $d$ dimensions (\ie the update family of $k$-neighbour model), we will sometimes write $I(d,k,\L)$ instead of
$I(\cU,\L)$ and we will say that $\L$ is $k$-internally spanned.  
\end{definition}
We will also need the following result  for $k$-neighbour bootstrap percolation
\begin{lemma}[\cite{CerfManzo} Lemma 4.1]
\label{rem:CerfManzo}
There exists $\epsilon>0$ s.t. for $L\ge C L_c(\epsilon q;\cU), \ L\in \bbN$ and $C$ a
large enough numerical constant,
\begin{equation}
  \label{eq:probspanned}
\mu(\o \in I(d,k,[L]^d))\ge 1- \exp(-L/L_c(\epsilon q;\cU)).
\end{equation}
\end{lemma}
Clearly, for any update family $\cU$, the following holds.  If $\o$ is such that the region $\L$ is $\cU$-internally spanned and $\o'$ is the configuration equal to zero in $\L$ and equal to $\o$ elsewhere, then there exists a legal decreasing canonical-path $\G_{\o,\o'}$ which only uses flips inside $\L$. In particular the length of $\G_{\o,\o'}$ is at most $|\L|$. By reversing the path we get a legal increasing path between $\o'$ and $\o$.

Before starting the actual proof, it will be useful to fix some
additional notation. Given the hypercube $\L=[n]^d$ and $i\in [d]$, we
set $E_i(\L)=\{x\in \L:\ x_j=1,\ j\neq i\}$ and we call it the $i^{th}$-{\sl edge} of $\L$.
%footnote
Strictly speaking an edge of $V$ is a set of the form $\{x\in V:\ x_j\in\{1,n\}\ \forall j\neq i\}$. Here we will only need edges with one end-point at the vertex $(1,\dots,1)$. Any $(d-1)$-dimensional set of the form $\L\cap \{x:x_i=j\}$, $j\in [n]$,
will be called an $i$-{\sl slice} and it will be denoted by
$Sl_j(\L;i)$. A generic $i$-{\sl frame} $F_j(\L;i)$, $j\in [n]$, is
the $(d-2)$-dimensional subset of $Sl_j(\L;i)$ consisting of the
vertices $x$ such that $x_{k}=1$ for some $k\neq i$. If $\L'= x+\L$
then $E_i(\L')=E_i(\L)+x$ etc. If clear from the context we will drop the specification $\L$ from the notation.

\subsection{Part I}
Here we define the blocks of the coarse-grained
analysis together with the good and super-good events and the mapping
$\Phi$. We do that separately for the FA-kf model and the GG-KCM.
\subsubsection{The FA-kf model with $ k\ge 3$} 
\label{k>2}Let $\ell=L_c(\epsilon q;k-1,d-1)$ (see \eqref{eq:Lc}) with $\epsilon$ defined in Remark \ref{rem:CerfManzo}
and fix
$n=A\ell \log \ell,$ with $A>2(d-1)+1$. 
\begin{definition}[$G_1,G_2,\Phi$]
The good event $G_1$ consists of all $\o\in S$ such that for all $i\in [d]$ every $i$-slice of $V$ is $(k-1)$-internally spanned.
The
super-good event $G_2$ consists of all $\o\in G_1$ such that the first slice in any direction is empty. The mapping $G_1\overset{\Phi}{\mapsto}
G_2$ is defined by $\Phi(\o)_x=0$ if $x\in \cup_{i=1}^{d}Sl_1(V;i)$ and
$\Phi(\o)_x=\o_x$ otherwise.
\end{definition}
With the triple $(G_1,G_2,\Phi)$ we get immediately that 
\begin{align*}
(1-p_1)&\le dn (1-\hat \mu(I(d-1,k-1,[n]^{d-1}))),\\  
p_2&=\hat\mu(G_2)\ge p_1q^{dn^{d-1}},\\ 
\l_\Phi &\le \left(\frac{2}{q}\right)^{dn^{d-1}}.  
\end{align*}
Using \eqref{eq:probspanned} together with the definition of $n$, we
get immediately that $1-p_1\le A\ell^{-(A-1)}\log\ell$ so that $\lim_{q\to
  0}(1-p_1)\bigl(\log(1/p_2)\bigr)^2=0$ for all
$A> 2d -1$.

\subsubsection{The FA-kf model with $k=2$}
\label{k=2}In this case we choose $V=\prod_{i\in [d]}[n_i]$ with $n_i =\bigl(\frac{ A } {q}\log
(1/q)\bigr)^{1/(d-1)}$ with $A>3/(d-1)$.
\begin{definition}[$G_1,G_2,\Phi$]
The good event $G_1$ consists of all $\o\in S$ such that, for all $i\in
[d]$ every $i$-slice of $V$ contains at least one empty vertex. The
super-good event $G_2$ consists of all $\o\in G_1$ such that any $i$-edge of $V$ is empty.
The mapping $G_1\overset{\Phi}{\mapsto}
G_2$ is defined by  $\Phi(\o)_x=0$ if $x\in \cup_{j=1}^{d}E_j$ and
$\Phi(\o)_x=\o_x$ otherwise.
\end{definition}
As before we easily get
\[
1-p_1=\hat\mu(G^c_1)\le dn (1-q)^{n^{d-1}} \le dn q^A,\quad 
p_2=\hat\mu(G_2)\ge q^{nd},\quad \l_\Phi\le \frac{2^{nd}}{q^{nd}},
\] 
where $2^{nd}$ is the number of possible configurations
$\o'\in \{0,1\}^{\cup_i E_i}$. 
In particular, for all $A> 3/(d-1)$,  $\lim_{q\to
  0}(1-p_1)\bigl(\log(1/p_2)\bigr)^2=0$.

 \subsubsection{The GG-KCM  model}
\label{section:GG}Here we choose $n_1=\lfloor\frac{A\log(1/q)}{q^2}\rfloor$ and 
$n_2= \lfloor\frac{A\log(1/q)}{q}\rfloor$, $A>6$. 
\begin{definition}
We say that
$\o\in G_1$ if all columns of $V=[n_1]\times [n_2]$ contain at least
one empty vertex  and all rows contain at least one pair
of adjacent empty vertices $(x,x')$. We say that
$\o\in G_2$ if $\o\in G_1$ and the first two adjacent columns of $V$ are empty. 
The mapping $\Phi$ is the one which empties the first two columns of $V$.  
\end{definition}
Again we easily obtain that 
\[
1-p_1= O\Bigl(q^{(A-2)/2}\log(1/q)\Bigr), \quad p_2=
O\Bigl(\exp\bigl[-\frac{2A}{q} \log(1/q)^{2}\bigr]\Bigr),\quad
\l_\Phi= O\bigl(2^{2n_2}/q^{2n_2}\bigr).
\]
so that $\lim_{q\to
  0}(1-p_1)\bigl(\log(1/p_2)\bigr)^2=0$ for $A>6$.

Notice that for all models the factor $\bigl(\l_\Phi/p_2^4\bigr)^d|V|$
appearing in Corollary \ref{rhs 9bis} is bounded from above by the
r.h.s. of \eqref{main1:1}  and \eqref{main1:2}. 
\subsection{Part II}
Here we complete the
proof of Theorem \ref{thm:main1} by defining the canonical-paths appearing in Lemmas \ref{A} and \ref{B}
  in such a way that: 
\begin{enumerate}[(a)]
\item they are legal canonical-paths;
\item  the congestion constants $\rho_A,\rho_B$ and the maximum length
  of the paths $N_A,N_B$ are such that $\max\left(\rho_AN_a,\rho_BN_B\right)$ is 
bounded from above by 
r.h.s. of \eqref{main1:1} for the FA-2f and the GG-KCM  models and by the r.h.s. of \eqref{main1:2} for the FA-kf model, $k\ge 3$. 
\end{enumerate}
A very useful strategy to carry out this program is based on the following simple result.
\begin{lemma}
\label{FA1f}
Fix $\o$ and let $\L_1,\L_2,\dots,\L_N$ be $N$ regions with the
property that, for any $j$ and $k=j\pm 1$, if we empty $\L_j$  then we
can also empty $\L_{k}$ by means of a legal schedule using only flips
in $\L_{k}$. Assume that $\o$ is such that $\L_1$ is empty and let
$\o'$ be obtained from $\o$ by emptying $\L_N$. Then there exists a
legal canonical-path $\G_{\o,\o'}=(\o^{(1)},\dots,\o^{(m)})$, $m\le 2\sum_i |\L_i|$, such that for any $j\in [m]$ the following holds.  If the configuration $\o^{(j+1)}$ is obtained from $\o^{(j)}$ by flipping a vertex in $\L_{k_j}$ then all the discrepancies (\ie the vertices where they differ) between $\o$ and $\o^{(j)}$ are contained in $\L_{k_j-1}\cup\L_{k_j}\cup\L_{k_j+1}$ if $k_j<N$ and in $\L_{N-1}\cup\L_{N}$ if $k_j=N$.    
\end{lemma}
\begin{proof}
By assumption we can first empty $\L_2$ and then $\L_3$ by using flips first in $\L_2$ and then in $\L_3$. Let $\eta$ be the new configuration and let $\s$ be the configuration obtained from $\o$ by emptying $\L_3$. We can then restore the original values of $\o$ in $\L_2$ by reversing the legal canonical-path $\G_{\s,\eta}$. Starting from $\s$ we can iteratively repeat the above procedure and get a final legal canonical-path $\G_{\o,\o'}$ with the required property.
\end{proof}
\begin{remark}
\label{conges}The fact that the discrepancies between an intermediate step of the
path $\o^{(j)}$ and
the starting configuration $\o$ are contained in a triple of
consecutive $\L_i$'s allows us to easily upperbound the congestion
constant $\rho_\G := \sup_{\tilde \o}\sumtwo{}{\o:\ \G_{\o,\o'}\ni \tilde
  \o}\frac{\mu(\o)}{\mu(\tilde \o)}    
 $ of the family $\{\G_{\o,\o'}\}_{\o\in S}$ 
by $\left(2/q\right)^{\max_{i}(|\L_{i-2}|+|\L_{i-1}|+|\L_i|)}$. This
  observation will be the main tool to bound the congestion constants
  $\rho_A,\rho_B$ appearing in Corollary \ref{rhs 9bis}.
\end{remark}
 \subsubsection{The FA-kf model with $k\ge 3$}
As before set $V=[n]^d$ with $n$ as in Section \ref{k>2}. The proof is
based on a series of simple observations which, under certain
natural assumptions, ensure the existence
of legal canonical-paths with some prescribed properties.
\begin{claim}
Let $\o$ be a configuration such that the $i$-slice $Sl_j(V;i)$ is empty and the $i$-slice $Sl_{j-1}(V;i)$  is $(k-1)$-internally spanned. Let $\omega'$ be such that $\o'\restriction Sl_{j-1}(V;i)=0$ and $\o'$ coincides with $\o$ elsewhere. Then there is a legal decreasing canonical-path $\G_{\o,\o'}$ which uses only flips inside $S_{j-1}(V;i)$. Similarly if we replace $S_{j-1}(V;i) $ with $S_{j+1}(V;i) $. 
\label{banale}
   \end{claim}
\begin{proof}
The  result can be immediately proven  by noticing that each site in $Sl_{j-1}(V;i)$ has an empty neighbour in $Sl_{j}(V;i)$. Since $Sl_{j-1}(V;i)$ is $(k-1)$-internally spanned, the legal (w.r.t. to the FA-(k-1)f constraint) monotone path which empties it is also legal w.r.t. the FA-kf constraint.
\end{proof}
\begin{claim} 
\label{claim1k>3}
 Fix $i\in[d],\, m\in [n]$ and let $(\o,\o')$ be a pair of
 configurations satisfying at least one of the following conditions: 
 \begin{enumerate}[(a)]
 \item $\omega$ is such that 
the first $i$-slice is empty and all the others are $(k-1)$-internally
spanned and $\omega'$ is obtained from $\o$ by emptying the
$m^{th}$ $i$-slice and the first $m-1$ $i$-frames.
\item $\omega$ is such
that $\cup_{i=1}^d Sl_1(V;i)$ is empty and $\omega'$ is
obtained from $\o$ by emptying $Sl_m(V;i)$. 
 \end{enumerate}
Then there exists a legal canonical-path
$\G_{\o,\o'}=(\omega^{(1)},\omega^{(2)},\dots,\omega^{(N)})$ with
$N\le 2n^d$ such that
the only discrepancies between $\o$ and $\omega^{(j)}$, $j\in [N]$, belong to the set 
\[
Sl_{k_j-1}(V,i)\cup Sl_{k_j}(V,i)\cup Sl_{k_j+1}(V;i)\cup \left(\cup_{\ell=1}^{k_j}F_\ell(V;i)\right),
\]
where $k_j$ is such that the flip connecting $\o^{(j)}$ to $\o^{(j+1)}$ occurs in the $k_j^{th}$ $i$-slice. 
\end{claim}
\begin{proof}
Case (a). In this case we simply apply Lemma \ref{FA1f} and Claim
\ref{banale} to the first $m$ $i$-slices with a twist. After emptying the $j^{th}$ $i$-slice, $j=1,2,\dots,m$, instead of reconstructing the original values of $\o$ in the previous slice we do so only in $Sl_{j-1}(V;i)\setminus F_{j-1}(V;i)$. In such a way the $i$-frames once emptied remain so and we get to the final configuration $\o'$ by a legal canonical-path satisfying the required property.
% For simplicity we prove the claim in the case $m=3$, the general case follows along the same lines. 
% By applying twice Claim  \ref{banale} 
% there exists a legal decreasing schedule which first empties
% $S_2(V;i)$ and then $S_3(V;i)$ by using only moves inside the second
% and third  $i$-slices. Denote by $\tilde\omega$ the configuration that
% has been reached in this way. Notice that $\omega'$ is the
% configuration which is empty in $F_2(V;i)$, coincides with $\omega$ in
% $S_2(V;i)\setminus F_2(V;i)$ and equals to $\tilde\omega$
% elsewhere. In particular the $2^{nd}$ $i$-slice is $(k-1)$-internally
% spanned for $\omega'$  while the third one is completely
% empty. Therefore there is a legal decreasing canonical-path $\G_{\o',\tilde\o}$
% which uses only flips in the $2^{nd}$ $i$-slice. It is sufficient to
% reverse this path and concatenate it with the one from $\o$ to $\tilde
% \o$ to get a legal canonical-path from
% $\omega$ to $\omega'$ satisfying the required property.

Case (b). We use again Lemma \ref{FA1f} and Claim \ref{banale}. The base
case $k=2, d=2$ follows by observing that the $i$-slices, $i=1,2$,
are $1$-internally spanned since they all contain an empty site. The
case $k=2$ and $d>2$ follows by induction. In fact $Sl_2(V;i)$ is of the form $\L
\times \{x_i=2\}$ with $\L$ isomorphic to $[n]^{d-1}$. Moreover
$\cup_{i=1}^{d-1}Sl_1(\L;j)\times \{x_i=2\}\subset \cup_{j=1}^d Sl_1(V;j)$ and therefore it
 is empty by assumption. By the inductive hypothesis for $k=2,d-1$ we
 can empty $Sl_2(V;i)$  using only flips inside $Sl_2(V;i)$. This
 concludes the proof for $k=2$ and any $d\ge 2$. 
We thus
assume the result true for $(k-1,d-1)$ and prove it for $(k,d)$, $d\ge
k$. In this case we apply Lemma \ref{FA1f} to the regions $\L_j:=Sl_j(V;i)\cup \left(\cup_{i=1}^d Sl_1(V;i)\right)$. For simplicity and w.l.o.g we only verify the assumption of the lemma for the pair $\L_1,\L_2$. In this case we aim at constructing a legal canonical-path that
empties $Sl_2(V;i)$ using only flips there.

Thus, using the inductive hypothesis and the fact that each site on $Sl_2(V;i)$ has an additional empty neighbour in $Sl_1(V;i)$, we can empty $Sl_2(V;i)$ 
by a legal canonical-path which uses flips only in $Sl_2(V;i)$.
\end{proof}
% \begin{claim}\label{claim2k>3}
%  Then there exists a legal canonical-path  $\G_{\o,\o'}=(\omega^{(1)},\omega^{(2)},\dots,\omega^{(N)})$ such that
% the only discrepancies between $\o$ and $\omega^{(j)}$, $j\in [N]$, belong to the set 
% \[
% S_{k_j-1}(V,i)\cup S_{k_j}(V,i)\cup \left(\cup_{\ell=1}^{k_j}F_\ell(V;i)\right),
% \]
% where $k_j$ is such that the flip connecting $\o^{(j)}$ to $\o^{(j+1)}$ occurs in the $k_j^{th}$ $i$-slice. 
% There exists a legal canonical-path $\G_{\o,\o'}$ $\omega^{(1)},\dots,\omega^{(N)}$ such that \begin{itemize} \item $\omega^{(1)}=\omega$, $\omega^{(N)}=\omega'$ \item 
% for each $j\in [1,N]$,  $\omega^{(j)}$ coincides with $\omega$ on all sites except on $S^i_{x^{(j)}_i-1}\cup  S^i_{x^{(j)}_i}\cup S^i_{x^{(j)}_i+1} \cup  _{j'=2}^{m-1} F_{j'}^i$.
%\end{itemize}
% \end{claim}
% \begin{proof}
% \end{proof}
We are now ready to state the main result for the case under consideration.
\begin{proposition}
\label{prop:k>2}
In the above setting there exists a choice of the canonical-paths occurring in Lemmas \ref{A} and \ref{B} such that, for a suitable positive constant $c$,
\begin{itemize}
\item each path is a legal canonical-path and $\max(N_A,N_B)\le c n^d$;  
\item $\max(\rho_A,\rho_B)\le (1/q)^{cn^{d-1}}$.
\end{itemize}
\end{proposition}
Using that $n=A\ell \log \ell$, $\ell$ being  the
critical length for the FA-(k-1)f model in $\bbZ^{d-1}$ given by (cf. \eqref{eq:Lc}) 
\[
\ell=\exp_{(k-2)}\Bigl(\frac{\l(d-1,k-1)+o(1)}{q^{1/(d-k+1)}}\Bigr), 
\]
the proposition implies that  
\[
\max(\rho_A N_A,\rho_B N_B)\le \text{r.h.s. of \eqref{main1:2}},
\]
so that the conclusion of Theorem \ref{thm:main1} for the case $k\ge 3$ follows from Corollary \ref{rhs 9bis}.\qed
\begin{proof}[Proof of the proposition]
We begin by examining the choice of the canonical-paths appearing in
Lemma \ref{B}. Using the definition of the good and super-good events
$G_1,G_2$ given in Section \ref{k>2}, our choice for the canonical
paths is the one dictated by (a) of Claim \ref{claim1k>3}. In this
case, using Remark \ref{conges}, $N_B\le c n^d$ and $\rho_B\le (1/q)^{ n^{d-1}}$ for some constant $c>0$. 

We now turn to the canonical-paths appearing in Lemma \ref{A}. Fix $\omega$ and $z$ as in the lemma and observe that, using (b) of claim \ref{claim1k>3}, we can empty all the slices $S_{z_i+1}(V;i)$, $i\in[d]$, via a legal schedule. Call $\o'$ the configuration obtained in this way. In $\o'$ we can make a flip at $z$ since $z$ has at least $d$ empty neighbors. We can finally reverse the path from $\o$ to $\o'$ to obtain our final legal canonical-path between $\o$ and $\o^z$. Claim \ref{claim1k>3} again implies that  
$N_A\rho_A\leq cn^{2d} 1/q^{cn^{d-1}}$.
\end{proof}
 \subsubsection{The FA-kf model with $k=2$}
As before set $V=[n]^d$ with $n$ as in Section \ref{k=2}.
For any $x\in V$ we define the \emph{cross at $x$} as the set $\mathcal C_x(V):=\cup_{i=1}^d \cC_x(V;i)$ with
\[
\cC_x(V;i):=\{x'\in V:\  x'_j=x_j\  \forall j\neq i\}.
\] 
Notice that the cross of the vertex $(1,1,\dots,1)\in V$ is the union
of the edges $E_i(V)$. 
\begin{claim}
\label{claim:cross}
Given $x,y \in V$ such that $y=x\pm \vec e_i$ for some $i\in [d]$, let $\omega$ be such that $\mathcal C_x(V)$ is empty
and let $\omega'$ be the configuration obtained from $\o$ by emptying
the cross at $y$. Then there exists a legal decreasing canonical-path
$\G_{\o,\o'}=(\omega^{(1)},\dots,\omega^{(m)})$, $m\le 2dn$, using
only flips in $\mathcal
C_x(V)\cup \cC_y(V)$. 
\end{claim}
\begin{proof}
Since $y=x+\pm \vec e_i$ then necessarily
$\cC_y(V;i)=\cC_x(V;i)$. Consider now the vertex $z=y\pm \vec e_j$
with $j\neq i$. This vertex has two empty neighbors: one is $y$ and
another belongs to $\cC_x(V)$. Therefore $z$ can be emptied. We can
iterate until we empty the $j^{th}$ arm of the cross $\cC_y(V)$ and
then repeat the procedure for all the remaining direction but the $i^{th}$-one. 
\end{proof}
As for the case $k\ge 3$ we have: 
\begin{proposition}
\label{prop:k=2}
In the above setting there exists a choice of the canonical-paths occurring in Lemmas \ref{A} and \ref{B} such that, for a suitable positive constant $c,$
\begin{itemize}
\item each path is a legal canonical-path and $\max(N_A,N_B)\le c n^2$;  
\item $\max(\rho_A,\rho_B)\le (1/q)^{cn}$.
\end{itemize}
\end{proposition}
Using that $n=\bigl(\frac{ A } {q}\log
(1/q)\bigr)^{1/(d-1)}$,
the proposition implies that  
\[
\max(\rho_A N_A,\rho_B N_B)\le \text{r.h.s. of \eqref{main1:1}},
\]
so that the conclusion of Theorem \ref{thm:main1} for the case $k=2$ follows from Corollary \ref{rhs 9bis}.\qed
\begin{proof}[Proof of Proposition \ref{prop:k=2}]
We begin by examining the choice of the canonical-paths appearing in
Lemma \ref{B}. Fix $\o$ and suppose that we have two hypercubes
$V=[n]^d$ and $V'=V +(n+1)\vec e_1$ such that $\o\restriction V$ is good
and $\o\restriction V'$ is super-good. Let also $\o'$ be obtained from
$\o$ by emptying the cross of the vertex $(1,1,\dots,1)\in V$ so that $\o'\restriction V$ is super-good. 
Let now $z^{(i)}$
be the first (according to some apriori order) vertex in the $(n-i+1)^{th}$ $1$-slice $Sl_{n-i+1}(V;1)$ which is empty and let
$\bar z^{(i)}=z^{(i)}+\vec e_1$. Observe that the vertex $\bar
z^{(i)}$ belong to the same $1$-slice of $V$ as the vertex $z^{(i-1)}$
and that the vertex $z^{(i)}$ exists for all
$i\in [n]$ because $\o\restriction V$ is
good. Finally let $\g=(x^{(1)},\dots x^{(m)})$, $m\le n^2$, be the geometric path connecting $x^{(1)}:=(1,\dots,1)+n\vec e_1\in V'$ with $x^{(m)}:=(1,\dots,1)\in V$ constructed according to the
following schedule:
\begin{enumerate}[(a)]
\item join $x^{(1)}$ with $\bar z^{(1)}$ by first adjusting the second coordinate,
then the third one etc;
\item join $\bar z^{(1)}$ to $z^{(1)}$;
\item repeat the above steps with $x^{(1)}$ replaced by $z^{(1)}$ and $\bar z^{(1)}$ by $\bar
  z^{(2)}$ etc.
\end{enumerate}
Next, for $i\in[m]$, let $\L_i$ be the cross $\cC_{x^{(i)}}(V^{(i)})$
where $V^{(i)}$ is the hypercube $V+ (x^{(i)}_1-1)\vec e_1$. Notice
that $x^{(i)}\in Sl_1(V^{(i)};1)$.
We claim that the above sets satisfy the assumption of Lemma
\ref{FA1f}. If the hypercubes $V^{(i)}, V^{(i+1)}$ are the same then
the claim follows immediately from Claim \ref{claim:cross}. If
$V^{(i+1)}=V^{(i)}-\vec e_1$ then necessarily the pair
$(x^{(i)},x^{(i+1)})$ must be of the form $(\bar z^{(j)},z^{(j)})$ for
some $j$ and having the cross $\cC_{x^{(i)}}(V^{(i)})$ empty implies
that also the cross $\cC_{x^{(i)}}(V^{(i+1)})$ is empty because, by
assumption, $\o_{z^{(j)}}=0$. Thus we can apply again Claim
\ref{claim:cross}, this time in the hypercube $V^{(i+1)}$, and empty
$\L_{i+1}$. It is now a simple check to verify that the path defined
in this way satisfy $N_B\le cn^2$ and $\rho_B\le e^{cn}$ for some
constant $c>0$. 

We now examine the canonical-paths entering in Lemma \ref{A}. Let $\o$
be such that all the hypercubes $V+\vec e_i$, $i\in[d]$, are
super-good, let $z\in V$ and let $\o'$ be obtained from $\o$ by
flipping $\o_z$. W.l.o.g. we assume in the sequel that
$z=(1,\dots,1)$. 

Let $\tilde \o$ be the intermediate configuration obtained from $\o$ by
emptying the cross (in $V$) of the vertex
$x^{(1)}:=(n,\dots,n)$. Using Lemma \ref{claim:cross} it is easy to
check that 
there exists a legal canonical-path $\G_{\o,\tilde \o}$
with a congestion constant $\rho_\G\leq (1/q)^{c n}$ for some constant
$c>0$. Next
let $\g=(x^{(1)},\dots,x^{(m)})$ be a geometric path connecting $x^{(1)}$
with the vertex $z+\sum_{i=1}^d \vec e_i$ and define
$\L_i=\cC_{x^{(i)}}(V)$. Using Claim \ref{claim:cross} and the
definition of $\tilde \o$ the sets $\{\L_i\}_{i=1}^m$ satisfy the
assumption of Lemma \ref{FA1f}. In conclusion we have proved the
existence of a legal canonical-path $\G_{\o,\hat\o}$ where $\hat \o$
is obtained from $\o$ by emptying the cross of $x^{(m)}$. Now we can
legally flip $z$ and then reverse the path
$\G_{\o,\hat \o}$ to finally get to $\o'=\o^z$. In conclusion we have obtained a legal canonical
path $\G_{\o,\o'}$ and the claimed properties of $N_A$ and $\rho_A$
follow at once from its explicit construction.   
\end{proof}
\subsubsection{The GG-KCM  model} 
Recall that in this case the basic block $V$ is the $[n_1]\times [n_2]$
rectangle, with $n_1,n_2$ as in Section \ref{section:GG}. Moreover,
given $\o\in \{0,1\}^V$, the block $V$ is good if every column
contains an empty site and every row contains a pair of adjacent empty
sites. It is super-good if it is good and the first two columns are empty.

In this setting two basic observations will be at the basis of our
definition of the canonical-paths appearing in
Lemmas \ref{A} and \ref{B}. 
Fix an integer $n$ together
with $\o\in \{0,1\}^{[4]\times [n+1]}$ and consider four consecutive columns $C_i=\{x=(i,j), \ j\in [n]\}$,
  $i\in [4]$.
\begin{enumerate}[(1)]
\item If $C_1,C_2$
  are empty and $C_3$ contains an empty site, then $C_3$ can be emptied
  by a legal decreasing canonical-path using only flips in
  $C_3$. Similarly if the role of $C_1$ and $C_3$ is interchanged.
\item If $C_1,C_2$
  are empty and the two vertices $x=(3,n+1)$ and $y=(4,n+1)$ above the
  $3^{th}$ and $4^{th}$ column are also empty, then $C_3$ and $C_4$ can be emptied
  by a legal decreasing canonical-path using only flips in $C_3\cup
  C_4$. Similarly if the role of the pair $(C_1,C_2)$ and
  $(C_3,C_4)$ is interchanged and the sites $x,y$ are replaced by $x'=(1,n+1),y'=(2,n+1)$. 
 \end{enumerate}
Using the above we can prove our final proposition.
\begin{proposition}
\label{prop:GG}
For the GG-KCM model there exists a choice of the canonical-paths occurring in Lemmas \ref{A} and \ref{B} such that, for a suitable positive constant $c,$
\begin{itemize}
\item each path is a legal canonical-path and $\max(N_A,N_B)\le c n_1n_2;$  
\item $\max(\rho_A,\rho_B)\le (1/q)^{cn_2}$.
\end{itemize}
\end{proposition}
\begin{proof}
We begin with the definition of the canonical-paths appearing in Lemma
\ref{B} with, for simplicity, $V_x=V$ and $V_y=V'$ where $V'$ is
either $V+(n_1+1)\vec e_1$ or
$V+(n_2+1)\vec e_2$. For simplicity we will not make any attempt to optimize our construction, \ie to improve over the constant $c$ above.
\begin{figure}
  \centering
\begin{tikzpicture}[>=latex]
%[x=1cm, y=1cm, semitransparent]
\draw [help lines] (0,0) grid [step=0.5] (14.5,3);
\begin{scope}
 \draw [thick]  (0,0) rectangle (7,3);
\node  at (3.5,-0.5) {$V$};
\draw [thick] (7.5,0) rectangle (14.5,3);
\node at (11,-0.5) {$V'$};
\draw[line width=4pt] (7.5,0) -- (7.5,3);
\draw[line width=4pt] (8,0) -- (8,3);
\draw[gray, line width=4pt] (3,0) -- (3,3);
\draw[gray,line width=4pt] (2.5,0) -- (2.5,3);
\draw[line width=4pt] (0,0) -- (0,3);
\draw[line width=4pt] (0.5,0) -- (0.5,3);
\foreach \i in {(3,2),(2.5,1),(1,0.5), (1.5,0.5), (1.5, 2),(2,2), (3.5,0), (4.5,0), (5,0), (5,3), (5.5,3), (4, 2.5),(4.5,2.5),(5,1.5), (5.5, 1.5),(6,1),(6.5,1), (7,2)} \node at \i
{\textbullet};
\draw [dashed,thick,->] (7.5,1.7)--(3.2,1.7);
\draw [dashed,thick,->] (2.3,1.7)--(0.75,1.7);
\end{scope}
\end{tikzpicture}
\caption{A sketch of the canonical-path $\G_{\o,\o'}$ appearing in
  Lemma \ref{B} for two horizontally
    adjacent blocks. Only the $1^{st}$  and $2^{nd}$
empty columns of the right super-good block are drawn (black).
The black dots in the left
block denote the empty sites, while the
    gray columns denote the different positions of the pair of adjacent columns
    inside the path. Notice the pair of adjacent empty
    sites on each row.  
  }
\label{fig:1}
\vskip 0.5cm
\begin{tikzpicture}[>=latex]
[x=1cm, y=1cm, semitransparent]
\draw [help lines] (0,0) grid [step=0.5] (7,6.5);
\begin{scope}
\draw [thick] (0,3.5) rectangle (7,6.5);
\node at (7.5,5) {$V'$};
\draw[line width=4pt] (0,3.5) -- (0,6.5);
\draw[line width=4pt] (0.5,3.5) -- (0.5,6.5);
\end{scope}
\begin{scope}
 \draw [thick]  (0,0) rectangle (7,3);
\node at (7.5,1.5) {$V$};
\draw[gray, line width=4pt] (2.5,1) -- (2.5,6.5);
\draw[gray, line width=4pt] (3,1) -- (3,6.5);
\draw[gray, line width=4pt] (5,3) -- (5,6.5);
\draw[gray, line width=4pt] (5.5,3) -- (5.5,6.5);
\draw[gray, line width=4pt] (3.5,2.5) -- (3.5,6.5);
\draw[gray, line width=4pt] (4,2.5) -- (4,6.5);
\draw[gray, line width=4pt] (1.5,2) -- (1.5,6.5);
\draw[gray, line width=4pt] (2,2) -- (2,6.5);
\draw[line width=4pt] (0,0) -- (0,3);
\draw[line width=4pt] (0.5,0) -- (0.5,3);
\draw[gray, line width=4pt] (6.5,1.5) -- (6.5,6.5);
\draw[gray, line width=4pt] (6,1.5) -- (6,6.5);
\foreach \i in {(3,2),(2.5,1),(1,0.5), (1.5,0.5), (1.5, 2),(2,2), (3.5,0), (4.5,0), (5,0), (5,3), (5.5,3), (3.5, 2.5),(4,2.5),(6,1.5), (6.5, 1.5),(3,1), (7,2)} \node at \i
{\textbullet};
\draw [dashed,thick,->] (0.7,6)--(4.9,6);
\draw [dashed,thick,->] (5,5)--(4.1,5);
\draw [dashed,thick,->] (3.4,4)--(2.1,4);
%\draw [dashed,ultra  thick,->] (5,5)--(4.2,5);
\draw [dashed,thick,->] (2.1,2.2)--(5.9,2.2);
\draw [dashed,thick,->] (5.8,1.5)--(3.1,1.5);
\node at (3.3,0.9) {$\mathbf \Lambda_5$};
\end{scope}
\end{tikzpicture}
  \caption{A sketch of the canonical-path $\G_{\o,\o'}$ for two vertically
    adjacent blocks. The sequence of the dashed arrows must be read from top to bottom. Initially the $1^{st}$ and $2^{nd}$ empty columns
    of the top block (drawn in thick black) travel until they sit above 
  the first pair of adjacent empty sites on the top row of the bottom
  block (grey position). At this time their height grows by one unit. Later in the
  path this new pair of empty columns is moved above
    the first pair of adjacent empty sites on the next to top row of
    the bottom block and so forth until the $1^{st}$ and $2^{nd}$
    columns of the bottom block become empty. 
  }
\label{fig:2}
\end{figure}

In the first case,
$V'=V+(n_1+1)\vec e_1$, let $\o\in \{0,1\}^{V\cup V'}$ be such that $V$ is good and $V'$
is super-good and let $\o'$ be obtained from $\o$ by emptying the first
two columns of $V$. Then we can use observation (1) above together with
Lemma \ref{FA1f} to get that there exists a legal canonical-path
$\G_{\o,\o'}$ of maximal length $c n_1n_2$ and congestion constant
$\rho_B\le (1/q)^{c n_2}$ for some constant $c>0$. Notice that in this
case we didn't use the fact that if $V$ is good then every row
contains a pair of adjacent empty sites (cf. Figure \ref{fig:1}).

In the second case,
$V'=V+(n_2+1)\vec e_2$, for $i\in [n]$ define $a_i$ as the smallest integer $j\in [n-1]$
such that 
$x=(j,n-i+1)$ and $y=(j+1,n-i+1)$ are both empty. Using that $V$ is
good the integer $a_i$ is well
defined. Let also $\L_i$ denotes the two semi-columns in $V\cup V'$ above the vertices
$(a_i,n-i+1)$ and $(a_i+1,n-i+1)$ (cf. Figure \ref{fig:2}).

Using observation (1) together with Lemma \ref{FA1f} we can then
obtain a legal canonical-path between $\o$ and $\o'$, whose length
is at most $c n_1n_2^2$ and whose congestion constant is bounded from
above by $(1/q)^{cn_2}$ for some $c>0$ independent of $i$, as follows:
\begin{enumerate}[(a)]
\item starting from the first two empty columns in $V'$, we begin to empty
  $\L_1$. Then, starting from the two empty semi-columns
  $\L_1\cup\{a_1,n\}\cup\{a_1+1,n\}$, we empty the two sites $x=(1,n),x'=(2,n)$ while
  restoring the original values of $\o$ in all the other sites of $V\cup V'$. 
\item We now repeat the same procedure with $\L_1$ replaced by $\L_2$ and $(x,x')$
 replaced by $\hat x=(1,n-1),\hat x'=(2,n-1)$, starting from the two
 empty semi-columns obtained by adding to the first two columns of
 $V'$ the empty sites $(1,n), (2,n)$.
\item We iterate until reaching $\o'$.  
\end{enumerate}
It remains to consider the construction of the canonical-paths
appearing in Lemma \ref{A} and for that we use both (1) and (2)
above. 

\begin{figure}
  \centering
\begin{tikzpicture}[>=latex]
[x=1cm, y=1cm, semitransparent]
\draw [help lines] (0,0) grid [step=0.5] (7,3);
\begin{scope}
\draw [thick] (7.5,0) rectangle (14.5,3);
\node at (11,-0.5) {$V_1$};
\draw[line width=4pt] (7.5,0) -- (7.5,3);
\draw[line width=4pt] (8,0) -- (8,3);
\end{scope}
\begin{scope}
\draw [thick] (0,3.5) rectangle (7,6.5);
\node at (2.5,4) {$V_2$};
\draw[line width=4pt] (0,3.5) -- (0,6.5);
\draw[line width=4pt] (0.5,3.5) -- (0.5,6.5);
\end{scope}
\begin{scope}
 \draw [thick]  (0,0) rectangle (7,3);
\node at (3.5,-0.5) {$V$};
\draw[dashed,line width=1.4pt] (3.75,5) ellipse (0.75 and 1.8);
\draw[line width=4pt] (3.5,3.5) -- (3.5,6.5);
\draw[line width=4pt] (4,3.5) -- (4,6.5);
\draw[line width=4pt] (4.5,0) -- (4.5,3);
\draw[line width=4pt] (5,0) -- (5,3);
\draw[dashed,line width=2.5pt] (4,0) -- (4,3);
%\draw[dashed,line width=2.5pt] (3.5,0) -- (3.5,3);
\draw[line width=4pt] (0,0) -- (0,3);
\draw[line width=4pt] (0.5,0) -- (0.5,3);
\draw [dashed,thick,->] (0.6,5)--(3.3,5);
\draw [dashed,thick,->] (7,1.5)--(5,1.5);
\draw [dashed,thick,->] (5,0.7)--(4.1,0.7);
\draw[dashed, thick,->] (3.8,1.5)--(0.8,1.5);
\node at (4.2,3.2) {$z$}; 
\node at (4,3) {\textbullet}; 
\node at (4.5,-0.5) {$x$};
\end{scope}
\end{tikzpicture}
  \caption{A sketch of the canonical-path $\G_{\o,\o'}$ appearing in Lemma \ref{A}. Assuming that the path has been able to empty the two black columns of $V$, then it is possible to move these two columns one step further to the left as follows. First move the initial pair of double empty columns in $V_2$ to the new position encircled by the dashed ellipse, then, starting with the vertex $z$, empty the dashed black column in $V$ and finally restore the original values of $\o$ to the right of $x$ and then in $V_2$. 
  }
\label{fig:3}
\end{figure}
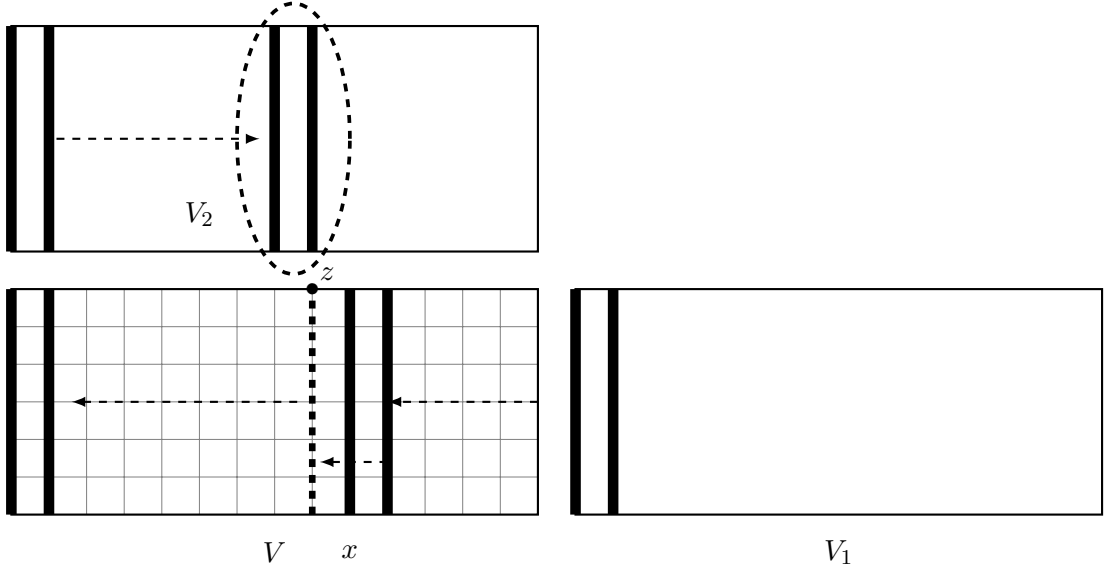
Fix $\o$ such that $V_1:=V+(n_1+1)\vec e_1$ and
$V_2:=V+(n_2+1)\vec e_2$ are super-good, let $z\in V$ and
let $\o'=\o^z$. For simplicity and w.l.o.g. we assume $z=(1,1)$.
We can then
obtain a legal canonical-path between $\o$ and $\o'$ with the required
properties as follows:
\begin{enumerate}[(a)]
\item by combining observation (1) with Lemma \ref{FA1f} we first empty the
  last two columns of $V_2$ without doing any flip inside $V\cup V_1$;
\item at this stage the last two columns of $V_2$ are empty because of (a) and the first two columns of $V_1$ are also empty because $V_1$ was super-good. Thus, using observation (2), we empty the last two columns of $V$;
\item finally we restore the original configuration in $V_2$ by reverting the
  path in the first step.
\item We repeat the above three  steps with a twist: we first empty
  the $4^{th}$  and $3^{rd}$ last column of $V_2$, then the $4^{th}$
  and $3^{rd}$ last column
  of $V$. We then restore the original configuration $\o$ in the last
  two columns of $V$ and, subsequently, we finally restore $\o$  in $V_2$. We have now reached
  the intermediate configuration obtained from $\o$ by emptying the $4^{th}$ and $3^{rd}$ last column of $V$. 
\item We iterate the above step until reaching the configuration obtained from $\o$
  by emptying the $2^{nd}$ and
  $3^{rd}$ column of $V$.
\item Finally, using again (2) above and Lemma \ref{FA1f}, we empty
  the vertex $(1,2)$. At this stage we can do a flip in the corner
  $(1,1)$ since the vertices $(1,2),(2,1)$ and $(3,1)$ are all empty. 
\item The final step is to retrace the steps of the path which emptied
  $(1,2)$ and then those of the path which emptied the $2^{nd}$ and $3^{rd}$ column of $V$ in such a way that we end up in the
  configuration $\o'$.       
\end{enumerate}

\end{proof}
\appendix
\section{Proof of Lemma \ref{lem:Paris}}
\noindent
Fix $\o\in \O$ and let $\tau_{BP}(\omega;x)$ denote the bootstrap
infection time of a generic site $x$ when the initial set of infected sites coincides
with the set of empty sites of
$\omega$. Given a sequence $\{x_i\}_{i=1}^n$ of vertices 
 of $\bbZ^d,$ we set  $\o^{(0)}=\omega$ and for any $i\in[n]$ we denote 
 by $\o^{(i)}$  the configuration
  obtained from $\o$ by setting equal to $0$  all the variables at $x_1,\dots,x_{i}.$ 
 We then say that $\{x_i\}_{i=1}^n$ is a {\sl sequence of legal infections}
 if, for all $i\in [n]$,
 \begin{enumerate}[(i)]
\item $\o_{x_i}=1$;
\item  there exists $U_i\in\mathcal U$ s.t. $\omega^{(i-1)}_y=0$ for all $y\in x_i+U_i$.
\end{enumerate}
Notice that, necessarily, $\t_{BP}(\o,x_j) \le j$ for $j\in [n]$.

The key \emph{reduction property} of a sequence of legal infections is that 
it is possible to extract a subsequence
$\{x_{i_k}\}_{k=1}^{m}$  of length $m=\tau_{BP}(\omega,x_n)$
%that is an $r$-path  ending at the origin, namely
such that $i_{m}=n$ and
\[
|x_{i_j}-x_{i_{j+1}}|\le r \quad \forall j\in [m-1],
\] 
with $r=\max_{X\in \cU}\max_{x\in X} |x|$.
Indeed, let 
 $t_i:=\t_{BP}(\o,x_i)$. If $t_n=1$ there is nothing to be proved. If $t_n>1$, then
 necessarily there exists $j<n$ such that 
\begin{itemize} 
\item $x_j$ belongs to one of the update rules of $x_n$
  and thus $|x_j-x_n|\leq r$, 
\item $t_j\geq  t_n-1$,
\end{itemize}
since otherwise $\t_{BP}(\o,x_n)<t_n$. Let $j_*$ be the largest
one such integers and set $i_{m-1}=j_*$. Then, using that
$\{x_i\}_{i=1}^{j_*}$ is also a sequence of legal infections, we can repeat the argument and proceed backward until identifying the claimed 
subsequence $\{x_{i_k}\}_{k=1}^m$.

%Consider now the KCM evolved starting from $\omega$ and let $\tau_0$ be the random time at which the origin has been emptied
%for the first time. Consider all the times at which some site has been updated  we denote by $s_<s_2<\dots<\tau_0$ all the times at which the sites have been updated to zero during the process

Let now $t\equiv \t_{BP}(\o,0)$ and fix $\d\in (0,1)$. In the sequel
it will be useful to think of the KCM dynamics as built according to the
standard graphical construction of an interacting particle system with
a Glauber dynamics (see e.g. \cite{CMRT}). In this setting, suppose that for any $r$-path $\g$ of length
$t$ ending at the origin, \ie a sequence of $t$
vertices $(v_1,\dots, v_n)$, with $v_t=0$ and $|v_{i}-v_{i+1}|\le r$
for all $i\in [t-1],$ it is not possible to find a ordered sequence  
$t_1<\dots< t_t$ in $(0,\d t)$ of rings of the Poisson
clocks such that the $i^{th}$-ring occurs
at $v_i$. Using the reduction property of any sequence of legal infections,
we conclude that, deterministically, the KCM dynamics starting from $\o$ cannot infect the origin within time $\d t$. % If Indeed, if $\tau_0$ is the random first time at 
% which the origin is set to 0 for the KCM,  the ordered sequence of sites that have been successively set to zero up to $\tau_0$
% must be a legal infecting sequence ending at the origin.

Finally we claim that the above assumption is satisfied w.h.p. if
$\d$ is small enough. In fact, for any given $r$-path $\g$ ending at
the origin, the probability that 
there exists an ordered sequence  
$t_1<\dots< t_t$ in $(0,\d t)$ as above, is just the probability that a Poisson random variable of
mean $\d t$ is larger than $t$. Since the number of such paths is
bounded by $e^{c(r)t}$, the claim follows immediately for $\d$ small
enough.

In conclusion, we have proved that there exists $\d>0$ such that, for any $\o$ such that
$\t_{BP}(\o,0)\ge 1$, 
\[
\bbP_\o(\t_0\ge \d \t_{BP}(\o,0))\ge 1/2.
\]
\qed

\section*{Acknowledgments}
We are deeply in debt to R. Morris for several enlightening and
stimulating discussions on bootstrap percolation models. We also acknowledge the
hospitality of our respective departments during several exchange visits and the
organizers of the 2016 Oberwolfach's workshop ``Large Scale Stochastic
Dynamics'' for their hospitality in a stimulating atmosphere.

\end{document}